\documentclass[a4paper,11pt]{article}
\usepackage{graphicx}
\usepackage[T1]{fontenc}
\usepackage[utf8]{inputenc}
\usepackage[english]{babel}
\usepackage{amsmath,amsthm,amssymb,amsfonts}
\usepackage{lmodern}
\usepackage{subcaption}
\usepackage{geometry}
\usepackage{framed}
 \usepackage{color}
\definecolor{darkgreen}{rgb}{0.00,0.5,0.00}
\usepackage{hyperref}
\hypersetup{
	colorlinks=true,        
	linkcolor=blue,         
	citecolor=darkgreen,         
	urlcolor=blue           
}

\usepackage{cite}
\usepackage{algorithmic}
\usepackage{algorithm}
\usepackage{enumerate}
\usepackage{empheq}
\usepackage{booktabs}
\newcommand{\ra}[1]{\renewcommand{\arraystretch}{#1}}

\newtheorem{theorem}{Theorem}
\theoremstyle{definition}
\newtheorem{lemma}{Lemma}

\newtheorem{remark}{Remark}

\newcommand{\n}[1]{\|#1 \|} 
\renewcommand{\a}{\alpha}
\renewcommand{\b}{\beta}
\renewcommand{\d}{\delta}

\newcommand{\la}{\lambda}

\renewcommand{\t}{\tau}
\newcommand{\s}{\sigma}
\newcommand{\e}{\varepsilon}
\renewcommand{\th}{\theta}
\newcommand{\p}{\varphi}

\newcommand{\R}{\mathbb R}
\newcommand{\E}{V}

\newcommand{\x}{\bar x}

\newcommand{\z}{\bar z}
\renewcommand{\u}{\mathbf{u}}

\newcommand{\lr}[1]{ \langle #1 \rangle}

\DeclareMathOperator{\prox}{prox}
\DeclareMathOperator{\dom}{dom}

\DeclareMathOperator{\id}{Id}

\DeclareMathOperator{\Fix}{Fix}
\DeclareMathOperator{\range}{range}

\DeclareMathOperator{\argmin}{argmin}
\DeclareMathOperator{\dist}{dist}

\title{Golden Ratio Algorithms for Variational Inequalities}

\author{Yura Malitsky\footnote{Institute for Numerical and
        Applied Mathematics, University of G\"ottingen,
        \href{mailto:y.malitsky@gmail.com}{y.malitsky@gmail.com}}}

\date{}

\begin{document}
\maketitle

\begin{abstract}
      The paper presents a fully adaptive algorithm for monotone
  variational inequalities. In each iteration the method uses two
  previous iterates for an approximation of the local Lipschitz
  constant without running a linesearch. Thus, every iteration of the
  method requires only one evaluation of a monotone operator $F$ and a
  proximal mapping $g$.  The operator $F$ need not be Lipschitz
  continuous, which also makes the algorithm interesting in the area
  of composite minimization. The method exhibits an ergodic $O(1/k)$
  convergence rate and $R$-linear rate under an error bound condition.
  We discuss possible applications of the method to fixed point
  problems as well as its different generalizations.
\end{abstract}

\medskip


\section{Introduction}
\label{sec:intro} We are interested in the variational inequality (VI) problem:
\begin{equation}
    \label{vip} \text{find } \quad z^*\in \E \quad \text{ s.t. } \quad \lr{F(z^*), z-z^*} +g(z)-g(z^*) \geq 0 \quad \forall z\in \E,
\end{equation} where $\E$ is a finite dimensional vector space and we assume that
\begin{enumerate}[ (C1) ] \label{C1-C3}
    \item \label{C1} the solution set $S$ of \eqref{vip} is nonempty;
    \item\label{C2} $g\colon \E \to (-\infty, +\infty]$ is a proper convex lower semicontinuous (lsc) function;
    \item \label{C3} $F\colon \dom g \to \E$ is monotone: $\lr{F(u)-F(v), u-v}\geq 0$ \quad $\forall u,v\in \dom g$.
\end{enumerate}
The function $g$ can be nonsmooth, and it is very common to consider
VI with $g=\d_C$, the indicator function of a closed convex set
$C$. In this case \eqref{vip} reduces to 
\begin{equation}
    \label{vip_C}
    \text{find } \quad z^*\in C \quad  \text{ s.t. } \quad \lr{F(z^*),
    z-z^*} \geq 0 \quad \forall z\in C,
\end{equation}
which is a more widely-studied problem. It is clear that one can rewrite \eqref{vip} as a
monotone inclusion: $0 \in (F + \partial g)(x^*)$. Henceforth, we
implicitly assume that we can (relatively simply) compute the
resolvent of $\partial g$ (the proximal operator of  $g$), that is
$(\id + \partial g)^{-1}$, but cannot do this for $F$, in other words computing
the resolvent $(\id +F)^{-1}$ is prohibitively expensive.

VI is a useful way to reduce many different problems that arise in
optimization, PDE, control theory, games theory to a common problem
\eqref{vip}. We recommend \cite{kinderlehrer1980introduction,pang:03}
as excellent references for a broader familiarity with the
subject.

As a motivation from the optimization point of view, we present two
sources where VI naturally arise. The first example is a
convex-concave saddle point problem:
\begin{equation}
    \label{eq:saddle} \min_x \max_y \mathcal{L}(x,y):= g_1(x) + K(x,y) - g_2(y),
\end{equation}
where $x\in \R^n$, $y\in \R^m$,
$g_1\colon \R^n \to (-\infty,+\infty]$,
$g_2\colon \R^m \to (-\infty,+\infty]$ are proper convex lsc functions
and $K\colon \dom g_1 \times \dom g_2\to \R$ is a smooth
convex-concave function. By writing down the first-order optimality
condition, it is easy to see that problem~\eqref{eq:saddle} is
equivalent to \eqref{vip} with $F$ and $g$ defined as
\begin{equation}
    z = (x,y) \qquad F(z) = \begin{bmatrix}{\nabla_x K(x,y)}\\{-\nabla_y K(x,y)}\end{bmatrix} \qquad g(z) = g_1(x) + g_2(y).
\end{equation}
Saddle point problems are ubiquitous in optimization as this is a very
convenient way to represent many nonsmooth problems, and this in turn
often allows to improve the complexity rates from $O(1/\sqrt k)$ to
$O(1/k)$.  Even in the simplest case when $K$ is bilinear form, the
saddle point problem is a typical example where the two simplest iterative
methods, the forward-backward method and the Arrow-Hurwicz method (see
\cite{arrow1958studies}), will not work. Korpelevich
in~\cite{korpel:76} and Popov in~\cite{popov:80} resolved this issue
by presenting two-step methods that converge for a general
monotone $F$. In turn, these two papers gave birth to various
improvements and extensions, see
\cite{nesterov2007dual,nemirovski2004prox,tseng00,mal-sem,malitsky15,reich:2011,lyashko11}.

Another important source of VI is a simpler problem of composite minimization
\begin{equation}\label{comp}
    \min_z J(z):= f(z) + g(z),
\end{equation}
where $z\in \R^n$, $f\colon \R^n\to \R$ is a convex smooth function
and $g$ satisfies \hyperref[C2]{(C2)}. This problem is also equivalent
to \eqref{vip} with $F = \nabla f$. On the one hand, it might not be
so clever to apply generic VI methods to a specific problem such as
\eqref{comp}. Overall, optimization methods have better theoretical
convergence rates, and this is natural, since they exploit the fact
that $F$ is a potential operator. However, this is only true under the
assumption that $\nabla f$ is $L$-Lipschitz continuous. Without such a
condition, theoretical rates for the (accelerated) proximal gradient
methods do not hold anymore. Recently, there has been a renewed
interest in optimization methods for the case with non-Lipschitz
$\nabla f$, we refer to
\cite{bauschke2016descent,tran2015composite,lu2018relatively}. An
interesting idea was developed in \cite{bauschke2016descent}, where
the descent lemma was extended to the more general case with Bregman
distance. This allows one to obtain a simple method with a fixed
stepsize even when $\nabla f$ is not Lipschitz. However, such results
are not generic as they depend drastically on problem instances
where one can use an appropriate Bregman distance.

For a general VI, even when $F$ is Lipschitz continuous but nonlinear,
computing its Lipschitz constant is not an easy task. Moreover, the
curvature of $F$ can be quite different, so the stepsizes governed by
the global Lipschitz constant will be very conservative. Thus, most
practical methods for VI are required to use linesearch --- an
auxiliary iterative procedure which runs in each iteration of the
algorithm until some criterion is satisfied, and it seems this is the
only option for the case when $F$ is not Lipschitz. To this end, most
known methods for VI with a fixed stepsize have their analogues with
linesearch. This is still an active area of research rich in diverse
ideas,
see~\cite{konnov1997class,tseng00,solodov1999new,solodov1996modified,iusem:97,bello2015variant,malitsky2018proximal,khobotov}. The
linesearch can be quite costly in general as it requires computing
additional values of $F$ or $\prox_g $, or even both in every
linesearch iteration. Moreover, the complexity estimates become not so
informative, as they only say how many outer iterations one needs to
reach the desired accuracy in which the number of linesearch
iterations is of course not included.

\bigskip

\textbf{Contributions.} In this paper, our aim is to propose an
\emph{adaptive} algorithm for solving problem~\eqref{vip} with $F$
locally Lipschitz continuous.  By \emph{adaptive} we mean that the
method does not require a linesearch to be run, and its stepsizes are
computed  using current information about the
iterates. These stepsizes approximate an inverse local Lipschitz
constant of $F$, thus they are separated from zero.  Each iteration of
the method needs only one evaluation of the proximal operator and one
value of $F$. Moreover, the stepsizes are allowed to increase from
iteration to iteration. To our knowledge, it is the first adaptive
method with such properties. The method is easy to implement and it
satisfies all standard rates for monotone VI: ergodic $O(1/k)$ and
$R$-linear if the error bound condition holds. In particular, one of
possible instances of the proposed algorithm can be given just in a
few lines
\begin{empheq}[box=\fbox]{align*}
  \la_k &= \min\Bigl\{ \frac{10}{9}\la_{k-1},\ \frac{9}{16 \la_{k-2}}\frac{\n{z^k-z^{k-1}}^2}{\n{F(z^k)-F(z^{k-1})}^2},\
    \bar \la\Bigr\}\\
  \z^k & = \frac{z^k+ 2 \z^{k-1}}{3} \\
z^{k+1}& = \prox_{\la_k g}(\z^{k} - \la_k F(z^k)),
\end{empheq}
where $z^0,z^1\in \E$, $\z^0 = z^1$, $\la_0=\la_{-1}>0$, $\bar \la
>0$.

Our approach is to start from the simple case when $F$ is
$L$-Lipschitz continuous. For this case, we present the \emph{Golden
    Ratio Algorithm} with a fixed stepsize, which is interesting on
his own right and gives us an intuition for the more difficult case
with dynamic steps. Sect.~\ref{sec:gold-ratio-algor} collects these
results. In Sect.~\ref{sec:fixed} we show how one can derive new
algorithms for fixed point problems based on the proposed
framework. In particular, instead of working with the standard class
of nonexpansive operators, we consider a more general class of
demi-contractive operators. Sect.~\ref{sec:general} collects two
extensions of the adaptive Golden Ratio Algorithm.  The first proposes
an extension of the adaptive algorithm enhanced by two auxiliary
metrics. Although it is simple theoretically, it is nevertheless still
very important in applications, where it is preferable to use
different weights for different coordinates. For the second extension
we do not need monotonicity of $F$, but instead we require that the
Minty variational inequality associated with \eqref{vip} has a
solution. In Sect.~\ref{sec:numer-exper} we illustrate the performance
of the method for several problems including the aforementioned
nonmonotone case. Finally, Sect.~\ref{sec:concl} concludes the paper
by presenting several directions for further research.  \bigskip

\textbf{Preliminaries.} Let $\E$ be a finite-dimensional vector space
equipped with inner product $\lr{\cdot,\cdot}$ and norm
$\n{\cdot} = \sqrt{\lr{\cdot,\cdot}}$. For a lsc function
$g\colon \E\to (-\infty, +\infty]$ by $\dom g$ we denote the domain of
$g$, i.e., the set $\{x \colon g(x)< + \infty\}$. Given a closed
convex set $C$, $P_C$ stands for the metric projection onto $C$,
$\d_C$ denotes the indicator function of $C$ and $\dist(x,C)$ the
distance from $x$ to $C$, that is $\dist(x,C) = \n{P_Cx - x}$.  The
proximal operator $\prox_g$ for a proper lsc convex function
$g\colon \E\to (-\infty, +\infty]$ is defined as
$\prox_g(z) = \argmin_x\{g(x)+\frac 1 2\n{x-z}^2\}$. The following
characteristic property (prox-inequality) will be frequently used:
\begin{equation}\label{prox_ineq}
    \x = \prox_{g}z \quad \Leftrightarrow \quad \lr{\x - z, x- \x} \geq g(\x)
    - g(x) \quad \forall x\in \E.
\end{equation}
A simple identity important in our analysis is
\begin{equation}\label{eq:id}
    \n{\a a + (1-\a)b}^2 = \a \n{a}^2 + (1-\a)\n{b}^2 - \a
    (1-\a)\n{a-b}^2 \ \  \forall a,b \in \E \ \forall \a \in \R.
\end{equation}
The following important lemma will simplify the proofs of the main
theorems.
\begin{lemma}\label{l:fejer}
  Let $(z^k) \subset \E$ be a bounded sequence  and suppose $\lim_{k\to \infty}\n{z^k-z}$ exists whenever $z$ is a cluster point of $(z^k)$. Then $(z^k)$ is convergent.
\end{lemma}
\begin{proof}
  Assume that $u, v$ are two arbitrary cluster points of $(z^k)$. From
  \begin{equation}
    \label{eq:for_lemma}
    \lr{z^k - u, u - v} = \frac 1 2 \n{z^k - v}^2 - \frac 1 2 \n{z^k-u}^2 - \frac 1 2 \n{u-v}^2
  \end{equation}
  we see that there exists $ \lim_{k\to \infty} \lr{z^k - u, u -
    v}$. Assume now that $z^{k_i}\to u$ and $z^{k_j}\to v$. Then one can observe that
  \begin{equation}
    \label{eq:for_lemma2}
    0 = \lim_{i \to\infty}\lr{z^{k_i}-u, u-v} = \lim_{j\to \infty} \lr{z^{k_j}-u,v-u} = \n{v-u}^2
  \end{equation}
and hence, $u = v$. Since $u,v$ were arbitrary, we can conclude that $(z^k)$ converges to some element in $\E$.  
\end{proof}

\section{Golden Ratio Algorithms}
\label{sec:gold-ratio-algor}

Let $\varphi=\frac{\sqrt 5 + 1}{2}$ be the golden ratio, that is
$\varphi^2 = 1 + \varphi$.
The proposed \emph{\underline{G}olden \underline{RA}tio
\underline{AL}gorithm} (GRAAL for short) reads as a simple recursion:
\begin{equation}\label{graal-1}
\begin{aligned}
  \z^k & = \frac{(\varphi-1)z^k+  \z^{k-1}}{\varphi}\\
z^{k+1}& = \prox_{\la g}(\z^{k}-\la F(z^k)).
\end{aligned}
\end{equation}
\begin{theorem}\label{th:graal}
    Suppose that $F\colon \dom g\to \E$ is $L$--Lipschitz and
    conditions \hyperref[C1]{(C1)--(C3)} are satisfied. Let
    $z^1,\z^0\in \E$ be arbitrary and
    $\la \in (0, \frac{\varphi}{2L}]$.  Then $(z^k)$, $(\z^k)$,
    generated by \eqref{graal-1}, converge to a solution of
    \eqref{vip}.
\end{theorem}
\begin{proof}
    By the prox-inequality~\eqref{prox_ineq} we have
\begin{equation}
    \label{eq:4}
    \lr{z^{k+1}-\z^k+\la F(z^k), z - z^{k+1}}\geq \la
    (g(z^{k+1})-g(z)) \quad \forall z\in \E
\end{equation}
and
\begin{equation}
    \label{eq:5}
    \lr{z^{k}-\z^{k-1}+\la F(z^{k-1}), z^{k+1} - z^{k}}\geq  \la( g(z^{k})-g(z^{k+1})).
\end{equation}
Note that $z^k - \z^{k-1} = \frac{1+\varphi}{\varphi}(z^k-\z^{k}) =
\varphi(z^k -\z^k)$. Hence, we can rewrite \eqref{eq:5} as
\begin{equation}
    \label{eq:6}
    \lr{\varphi(z^{k}-\z^{k})+\la F(z^{k-1}), z^{k+1} - z^{k}}\geq \la( g(z^{k})-g(z^{k+1})).
\end{equation}
Summation of \eqref{eq:4} and \eqref{eq:6} yields
\begin{multline}\label{eq:summation}
    \lr{z^{k+1}-\z^k, z - z^{k+1}} + \varphi \lr{z^{k}-\z^{k}, z^{k+1}
    - z^{k}} + \la \lr{F(z^k),z - z^k} \\ + \la \lr{F(z^k)-F(z^{k-1}), z^k - z^{k+1}} \geq \la (g(z^k) - g(z)).
  \end{multline} 
Expressing the first two terms in~\eqref{eq:summation} through norms,  we arrive at
\begin{align}
    \label{eq:7}
    \n{z^{k+1}-z}^2  \leq \n{\z^k-z}^2 & - \n{z^{k+1}-\z^k}^2 +2\la\lr{F(z^k)-F(z^{k-1}),z^k-z^{k+1}} \nonumber \\ & +\varphi(\n{z^{k+1}-\z^k}^2 - \n{z^{k+1}-z^k}^2
    -\n{z^k-\z^k}^2) \nonumber \\ 
  & -2\la\bigl(\lr{F(z^k),z^k-z} +g(z^k)-g(z)\bigr).
\end{align}
Choose $z=z^*\in S$. By \hyperref[C3]{(C3)}, the rightmost term in
\eqref{eq:7} is nonnegative:
\begin{equation}
    \label{eq:monot-est}
    \lr{F(z^k),z^k-z^*} + g(z^k)-g(z^*) \geq \lr{F(z^*),z^k-z^*} + g(z^k)-g(z^*)
     \geq 0.
\end{equation}
By \eqref{eq:id} and $z^{k+1} = (1+\varphi)\z^{k+1}- \varphi \z^{k} $ we have
\begin{align}\label{eq:identity}
    \n{z^{k+1}-z^*}^2&  =(1+\varphi)\n{\z^{k+1}-z^*}^2 -
    \varphi \n{\z^{k}-z^*}^2+\varphi(1+
    \varphi)\n{\z^{k+1}-\z^k}^2\nonumber \\
    & = (1+\varphi)\n{\z^{k+1}-z^*}^2 -\varphi \n{\z^{k}-z^*}^2+\frac{1}{\varphi}\n{z^{k+1}-\z^k}^2.
\end{align}
Combining \eqref{eq:7} and \eqref{eq:monot-est} with \eqref{eq:identity}, we deduce
\begin{align}
    \label{eq:8}
(1+\varphi)\n{\z^{k+1}-z^*}^2 \leq
    (1 & +\varphi)\n{\z^{k}-z^*}^2  - \varphi (
    \n{z^{k+1}-z^k}^2+\n{z^k-\z^k}^2) \nonumber \\
    & +2\la\n{F(z^k)-F(z^{k-1})}\n{z^k-z^{k+1}}.
\end{align}
From $\la \leq \frac{\varphi}{2L}$ it follows that
\begin{equation}
2 \la \n{F(z^k)-F(z^{k-1})}\n{z^{k+1}-z^k}\leq
\frac{\varphi}{2}(\n{z^{k}-z^{k-1}}^2 +\n{z^{k+1}-z^k}^2),
\end{equation}
which finally leads to
\begin{equation}
    \label{eq:9}
    (1+\varphi)\n{\z^{k+1}
-z^*}^2 + \frac{\varphi}{2}\n{z^{k+1}-z^k}^2  \leq
    (1+\varphi)\n{\z^{k}-z^*}^2 +  \frac{\varphi}{2}\n{z^{k}-z^{k-1}}^2
    - \varphi \n{z^k-\z^k}^2.
\end{equation}
From \eqref{eq:9} we have that $\bigl((1 + \varphi)\n{\z^{k+1} - z^*}^2 + \frac{\varphi}{2}\n{z^{k+1} - z^k}^2\bigr)$ is bounded and so is  $(\z^k)$, and
$\lim_{k\to \infty} \n{z^k-\z^k} = 0$. Hence, $(z^k)$ has at least one
cluster point. By definition of $\z^k$, $z^{k+1}-\z^{k} = \varphi (z^{k+1}-\z^{k+1}) \to 0$ and, hence, $z^{k+1}-z^k\to 0$
as well.  Taking the limit in \eqref{eq:4} (going to the subsequences
if needed) and using \hyperref[C2]{(C2)}, we prove that all cluster
points of $(z^k)$ (and thus, of $(\z^k)$) belong to $S$.  From
\eqref{eq:9} one can see that the sequence
$\bigl((1+\varphi)\n{\z^k-z^*}^2 + \frac{\varphi}{2}\n{z^k-z^{k-1}}\bigr)$ is
non-increasing and, hence, it is convergent. As
$\lim_{k\to \infty}\n{z^k-z^{k-1}} = 0$, there must exist
$\lim_{k\to \infty}\n{\z^k - z^*}$. Since $z^*$ is an arbitrary point
in $S$, from Lemma~\ref{l:fejer} it follows that $(\z^k)$ and $(z^k)$
converge to a point in $S$.   
\end{proof}

Notice that the constant $\varphi$ is chosen not arbitrary, but as the
largest constant $c$ that satisfies $\frac 1 c \geq c - 1$ in order to
get rid of the term $\n{z^{k+1}-\z^k}^2$ in \eqref{eq:7}. It is interesting to compare
the proposed GRAAL with the reflected projected (proximal) gradient
method \cite{malitsky15}. At a first glance, they are quite similar:
both need one $F$ and one $\prox_g$ per iteration. The advantage of
the former, however, is that $F$ is computed at $z^k$, which is always
feasible $z^k\in \dom g$, due to the properties of the proximal
operator. In the reflected projected (proximal) gradient method $F$ is
computed at $2z^k - z^{k-1}$ which might be infeasible. Sometimes, as
it will be illustrated in Sect.~\ref{sec:numer-exper}, this can be
important.

\subsection{Adaptive Golden Ratio Algorithm}
\label{sec:linesearch}

In this section we introduce our fully adaptive algorithm.  The
algorithm preserves the same computational cost per iteration as
\eqref{graal-1} (i.e., no linesearch) and the stepsizes approximate an
inverse local Lipschitz constant of $F$. Furthermore, for our
purposes, the locally Lipschitz continuity of $F$ is sufficient.  The
algorithm, which we call the adaptive Golden Ratio Algorithm (aGRAAL),
is presented as Alg.~\ref{alg:egraal}. For simplicity, we adopt the
convention $\frac{0}{0} = +\infty$.

\begin{algorithm}[!ht]\caption{\textit{Adaptive Golden Ratio Algorithm
    }}
    \label{alg:egraal}
\begin{algorithmic}
   \STATE {\bfseries Input:}
Choose $z^0, z^1 \in \E$, $\la_0>0$, $\phi \in (1,\varphi]$, $\bar \la
>0$. Set $\z^0 = z^1$, $\th_0 = 1$, $\rho = \frac{1}{\phi} + \frac{1}{\phi^2}$.
\smallskip
\STATE {\bfseries For $k\geq 1 $ do}
\STATE \quad  1. Find the stepsize:
\begin{equation}
    \label{la}
    \la_k = \min\Bigl\{ \rho \la_{k-1},\ \frac{\phi
    \th_{k-1}}{4\la_{k-1}}\frac{\n{z^k-z^{k-1}}^2}{\n{F(z^k)-F(z^{k-1})}^2},\
    \bar \la\Bigr\}
\end{equation}
\STATE \quad 2. Compute the  next iterates:
\begin{align}
  \z^k & = \frac{(\phi -1)z^k+ \z^{k-1}}{\phi} \label{zk}\\
z^{k+1}& = \prox_{\la_k g}(\z^{k}-\la_k F(z^k)). \label{zk+1}
\end{align}

\STATE \quad 3. Update: $\th_k = \dfrac{\la_k}{\la_{k-1}} \phi $.
\end{algorithmic}
\end{algorithm}

Notice that in the $k$-th iteration we need only to compute $F(z^k)$
once and reuse the already computed value $F(z^{k-1})$. The constant
$\bar \la$ in \eqref{la} is given only to ensure that $(\la_k)$ is
bounded. Hence, it makes sense to choose $\bar \la$ quite large.
Although $\la_0$ can be arbitrary from the theoretical point of view,
from \eqref{la} it is clear that $\la_0$ will influence further steps,
so in practice we do not want to take it too small or too large. The
simplest way is to choose $z^0$ as a small perturbation of $z^1$ and
take $\la_0 = \frac{\n{z^1-z^0}}{\n{F(z^1)-F(z^0)}}$. This gives us an
approximation\footnote{We assume that $F(z^1)\neq F(z^0)$, otherwise
    choose another $z^0$.} of the local inverse Lipschitz constant of
$F$ at $z^1$.

As one can see from Alg.~\ref{alg:egraal}, for $\phi< \varphi$ one has
$\rho>1$ and hence, $\la_k$ can be larger than $\la_{k-1}$. This is
probably the most important feature of the proposed algorithm. When
$F$ is Lipschitz, there are some other methods without linesearch that
do not require to know the Lipschitz constant,
e.g.,\cite{yang2018modified}. However, such methods use nonincreasing
stepsizes, which is rather restrictive.

Condition~\eqref{la} leads to two key estimations. First, one has
$\la_k \leq \la_{k-1} (\frac{1}{\phi}+\frac{1}{\phi^2})$, which in
turn implies  $\th_k -1 -\frac{1}{\phi}\leq 0$. Second, from $\frac{\phi
\th_{k-1}}{\la_{k-1}} = \frac{\th_k \th_{k-1}}{\la_k}$ one can derive
\begin{equation}
    \label{la-est2}
    \la_k^2\n{F(z^k)-F(z^{k-1})}^2 \leq \frac{\th_k \th_{k-1}}{4} \n{z^k-z^{k-1}}^2.
\end{equation}

Alg.~\ref{alg:egraal} is quite generic, and in order to simplify it we
can predefine some of the constants. Let $\phi = \frac 3 2$. Then
$\rho = \frac{10}{9}$ and as
$\th_{k-1} = \frac{\la_{k-1}}{\la_{k-2}}\phi$, we have
$\frac{\phi \th_{k-1}}{4\la_{k-1}} = \frac{9}{16\la_{k-2}}$. It is
easy to see that with such choice of constants, Alg.~\ref{alg:egraal}
reduces to the one we presented in the introduction.

\begin{lemma}\label{l:separ_la1}
  Suppose that $F\colon \dom g \to \E$ is locally Lipschitz continuous. If the sequence
  $(z^k)$ generated by Alg.~\ref{alg:egraal} is bounded, then both
  $(\la_k)$ and $(\th_k)$ are bounded and separated from $0$.
\end{lemma}

\begin{proof}
    It is obvious that $(\la_k)$ is bounded. Let us prove by induction
    that it is separated from $0$. As $(z^k)$ is bounded there exists
    some $L>0$ such that
    $\n{F(z^k)-F(z^{k-1})}\leq L \n{z^k-z^{k-1}}$. Moreover, we can
    take $L$ large enough to ensure that
    $\la_i\geq \frac{\phi^2}{4L^2 \bar \la}$ for $i=0,1$. Now suppose
    that for all $i=0,\dots k-1$,
    $\la_{i}\geq \frac{\phi^2}{4L^2 \bar \la}$. Then we have either
    $\la_k = \rho\la_{k-1} \geq \la_{k-1} \geq \frac{\phi^2}{4L^2 \bar \la}$ or
\begin{equation}
    \la_k =
    \frac{\phi^2}{4\la_{k-2}}\frac{\n{z^k-z^{k-1}}^2}{\n{F(z^k)-F(z^{k-1})}^2}\geq
    \frac{\phi^2}{4\la_{k-2}L^2} \geq \frac{\phi^2}{4L^2 \bar \la}.
\end{equation}
Hence, in both cases $\la_{k}\geq \frac{\phi^2}{4L^2 \bar \la}$. The
claim that $(\th_k)$ is bounded and separated from $0$ now follows
immediately.  
\end{proof}

Define the bifunction $\Psi(u,v):=\lr{F(u),v-u}+g(v)-g(u)$.  It is
clear that \eqref{vip} is equivalent to the following equilibrium
problem: find $z^*\in \E$ such that $\Psi(z^*,z)\geq 0$
$\forall z\in \E$. Notice that for any fixed $z$, the function
$\Psi(z, \cdot)$ is convex.

\begin{theorem}\label{th:egraal}
    Suppose that $F\colon \dom g \to \E$ is locally Lipschitz and
    conditions \hyperref[C1]{(C1)--(C3)} are satisfied. Then $(z^k)$
    and $(\z^k)$, generated by Alg.~\ref{alg:egraal}, converge to a
    solution of~\eqref{vip}.
\end{theorem}
\begin{proof}
    Let $z\in \E$ be arbitrary.  By the prox-inequality
    \eqref{prox_ineq} we have
\begin{align}
   \lr{z^{k+1}-\z^k+\la_k F(z^k), z - z^{k+1}}  & \geq \la_k
    (g(z^{k+1})-g(z))  \label{ls:eq:4}\\
  \lr{z^{k}-\z^{k-1}+\la_{k-1} F(z^{k-1}), z^{k+1} - z^{k}}  &\geq  \la_{k-1}( g(z^{k})-g(z^{k+1}))\    \label{ls:eq:5}.
\end{align}

Multiplying \eqref{ls:eq:5} by $\frac{\la_k}{\la_{k-1}} \geq
0$ and using that $\frac{\la_k}{\la_{k-1}}  (z^k-\z^{k-1})=\th_k(z^k-\z^{k})$, we obtain
\begin{equation}
    \label{ls:eq:6}
    \lr{\th_k(z^{k}-\z^{k})+\la_k F(z^{k-1}), z^{k+1} - z^{k}}\geq
    \la_k ( g(z^{k})-g(z^{k+1})).
\end{equation}
Summation of~\eqref{ls:eq:4} and~\eqref{ls:eq:6} gives us
\begin{align}
    \label{ls:eq:7}
    \lr{z^{k+1}-\z^k, z - z^{k+1}} & + \th_k \lr{z^{k}-\z^{k}, z^{k+1}
        -z^{k}}  + \la_k \lr{F(z^k)-F(z^{k-1}), z^k-z^{k+1}} \nonumber
  \\ & \geq  \la_k \lr{F(z^k), z^k-z} + \la_k (g(z^{k})-g(z))
       \nonumber \\ & \geq \la_k [
    \lr{F(z), z^k-z} + g(z^{k})-g(z)] = \la_k \Psi(z, z^k)    .
\end{align}
Expressing the first two terms in~\eqref{ls:eq:7} through norms, we derive
\begin{align}
    \label{ls:eq:8}
  \n{z^{k+1}-z}^2 \leq \n{\z^k-z}^2 & - \n{z^{k+1}-\z^k}^2
                                      +2\la_k\lr{F(z^k)-F(z^{k-1}),z^k-z^{k+1}}\nonumber
  \\  & +  \th_k \bigl(\n{z^{k+1}-\z^k}^2 -
  \n{z^{k+1}-z^k}^2-\n{z^k-\z^k}^2\bigr) - 2\la_k \Psi(z, z^k).
\end{align}  
Similarly to~\eqref{eq:identity}, we have
\begin{equation}\label{ls:eq:identity}
    \n{z^{k+1}-z}^2 = \frac{\phi}{\phi-1}\n{\z^{k+1}-z}^2 -
    \frac{1}{\phi-1}\n{\z^{k}-z}^2+\frac{1}{\phi}\n{z^{k+1}-\z^k}^2.
\end{equation}
Combining this with~\eqref{ls:eq:8}, we obtain
\begin{multline}
    \label{ls:eq:9}
    \frac{\phi}{\phi -1 } \n{\z^{k+1}-z}^2  \leq  \frac{\phi}{\phi -1 }
    \n{\z^{k}-z}^2 + \bigl(\th_k - 1 -
    \frac{1}{\phi}\bigr) \n{z^{k+1}-\z^{k}}^2   - 2\la_k \Psi(z, z^k) \\
    -  \th_k \bigl(\n{z^{k+1}-z^k}^2 + \n{z^k-\z^k}^2\bigr) +
    2\la_k\lr{F(z^k)-F(z^{k-1}),z^k-z^{k+1}}.
\end{multline}
Recall that $\th_k\leq 1 + \frac 1 \phi$. Using \eqref{la-est2}, the
rightmost term in~\eqref{ls:eq:9} can be estimated as
\begin{multline}
    \label{ls:eq:10}
     2\la_k\lr{F(z^k)-F(z^{k-1}),z^k-z^{k+1}} \leq
     2\la_k\n{F(z^k)-F(z^{k-1})}\n{z^k-z^{k+1}} \\ \leq \sqrt{\th_k
     \th_{k-1}}\n{z^k-z^{k-1}}\n{z^k-z^{k+1}} \leq
     \frac{\th_k}{2} \n{z^{k+1}-z^k}^2 +  \frac{\th_{k-1}}{2} \n{z^{k}-z^{k-1}}^2.
\end{multline}
Applying the obtained estimation to~\eqref{ls:eq:9}, we deduce
\begin{multline}
    \label{ls:eq:11}
    \frac{\phi}{\phi -1 } \n{\z^{k+1}-z}^2  +
    \frac{\th_k}{2}\n{z^{k+1}-z^k}^2 +  2\la_k \Psi(z, z^k) \\ \leq  \frac{\phi}{\phi -1 }
    \n{\z^{k}-z}^2 +  \frac{\th_{k-1}}{2}\n{z^{k}-z^{k-1}}^2
 -\th_k \n{z^k-\z^k}^2.
\end{multline}
Iterating the above inequality, we derive
\begin{multline}
    \label{ls:eq:12}
    \frac{\phi}{\phi -1 } \n{\z^{k+1}-z}^2 +
    \frac{\th_k}{2}\n{z^{k+1}-z^k}^2 + \sum_{i=2}^k\th_i
    \n{z^i-\z^{i}}^2 + 2\sum_{i=1}^k \la_k \Psi(z, z^k) \\ \leq
    \frac{\phi}{\phi -1 } \n{\z^{2}-z}^2 +
    \frac{\th_{1}}{2}\n{z^{2}-z^{1}}^2 -\th_2 \n{z^2-\z^2}^2 + 2\la_1
    \Psi(z, z^1).
\end{multline}
Let $z=z^*\in S$. Then the last term in the left-hand side of
\eqref{ls:eq:12} is nonnegative.  This yields that $(\z^k)$, and hence
$(z^k)$, is bounded, and $\th_k \n{z^k-\z^{k}}\to 0$. Now we can apply
Lemma~\ref{l:separ_la1} to deduce that
$\la_k\geq \frac{\phi^2}{4L^2 \bar{\la}}$ and $(\th_k)$ is separated
from zero. Thus, $\lim_{k\to \infty}\n{z^k-\z^k} = 0$, which implies
$z^k-\z^{k-1} \to 0$ and thus, also $z^{k+1}-z^k \to 0$. Let us show that all cluster points of $(z^k)$
and $(\z^k)$ belong to $S$. This is proved in the standard way. Let
$(k_i)$ be a subsequence such that $z^{k_i}\to \tilde z$ and
$\la_{k_i}\to \la >0$ as $i\to \infty$. Clearly,
$z^{k_i+1}\to \tilde z$ and $\z^{k_i}\to \tilde z$ as well. Then
consider \eqref{ls:eq:4} indexed by $k_i$ instead of $k$.  Taking the
limit as $i\to \infty$ in it and using (C2), we obtain
\begin{equation}
    \la \lr{F(\tilde z), z - \tilde z}  \geq \la (g(\tilde z)-g(z))
    \quad \forall z\in \E.
\end{equation}
Hence,  $\tilde z \in S$. From \eqref{ls:eq:11} one can see that the sequence $\bigl(\frac{\phi}{\phi -1 }
    \n{\z^{k}-z}^2 +  \frac{\th_{k-1}}{2}\n{z^{k}-z^{k-1}}^2\bigr)$ is nonincreasing, hence $\lim_{k\to \infty}\n{\z^k-z}$ exists. As $z\in S$ is arbitrary, from Lemma~\ref{l:fejer} it follows that sequences $(\z^k)$, $(z^k)$ converge to some element in $S$. This completes the proof. 
\end{proof}

\begin{remark} \label{rem:particular}
For a variational inequality in the form~\eqref{vip_C}, condition~\hyperref[C1]{(C3)} can be relaxed to the following
\begin{enumerate}[  (C4)  ] \label{C4}
    \item   \qquad \qquad \qquad \qquad      $   \lr{F(z), z-\z}\geq 0 \quad \forall z\in C \quad \forall \z\in S$.
\end{enumerate}
This condition, used for example in \cite{solodov1999new}, is weaker
than the standard monotonicity assumption~\hyperref[C3]{(C3)} or
pseudomonotonicity assumption:
\begin{equation}
  \lr{F(v), u-v}\geq 0 \quad \Longrightarrow \quad \lr{F(u), u -v}\geq 0.
\end{equation}
It is straightforward to see that Theorems~\ref{th:graal}, \ref{th:egraal} hold under
\hyperref[C4]{(C4)}. In fact, in the proof of the theorems, we choose
$z = z^*\in S$ only to ensure that $\Psi(z^*, z^k)\geq 0$. However, it
is sufficient to show that the first left-hand side in~\eqref{ls:eq:7} is
nonnegative. But for $z = z^*\in S$ this is true, since by \hyperref[C4]{(C4)},
$\lr{F(z^k), z^k-z^*}\geq 0$.
\end{remark}

\subsection{Ergodic convergence}
It is known that many algorithms for monotone VI (or for more specific
convex-concave saddle point problems) exhibit an $O(1/k)$ rate of
convergence, see, for instance,
\cite{nemirovski2004prox,nesterov2007dual,monteiro2011complexity,pock:ergodic},
where such an ergodic rate was established. Moreover, Nemirovski has
shown in \cite{nemirovsky1992information,nemirovski2004prox} that this
rate is optimal. In this section we prove the same result for our
algorithm.  When the set $\dom g$ is bounded establishing such a rate
is a simple task for most methods, including aGRAAL, however the case
when $\dom g$ is unbounded has to be examined more carefully. To deal
with it, we use the notion of the restricted merit function, first
proposed in \cite{nesterov2007dual}.

Choose any $\tilde x \in \dom g$ and $r> 0$. Let $U =
\dom g \cap \mathbb{B}(\tilde x,r)$, where $\mathbb{B}(\tilde x,r)$
denotes a closed ball with center $\tilde x$ and radius $r$. Recall the bifunction we work with
    $\Psi(u,v):=\lr{F(u),v-u}+g(v)-g(u)$. The restricted merit (dual gap)
function is defined as
\begin{equation}
    \label{eq:gap}
    e_r(v) = \max_{u\in U} \Psi(u,v) \qquad \forall v \in \E.
\end{equation}
From \cite{nesterov2007dual} we have the following fact:
\begin{lemma}
    \label{lemma:gap}
    The function $e_r$ is well defined and convex on $\E$. For any
    $x\in U$, $e_r(x)\geq 0$. If $x^*\in U$ is a solution to
    \eqref{vip}, then $e_r(x^*) = 0$. Conversely, if $e_r(\hat x) = 0$
    for some $\hat x$ with $\n{\hat x - \tilde x} < r$, then $\hat x$
    is a solution of \eqref{vip}.
\end{lemma}
\begin{proof}
    The proof is almost identical to Lemma 1 in
    \cite{nesterov2007dual}. The only difference is that we have to
    consider  VI \eqref{vip} with a general $g$ instead of $\d_C$. 
\end{proof}

Now we can obtain something meaningful. Since $F$ is continuous and
$g$ is lsc, there exist some constant $M>0$ that majorizes the
right-hand side of \eqref{ls:eq:12} for all $z\in U$. From this follows that
$\sum_{i=1}^k \la_i \Psi(z, z^i) \leq M$ for all $z\in U$ (we ignore
the constant $2$ before the sum). Let $Z^k$ be the ergodic sequence:
$Z^k = \frac{\sum_{i=1}^k \la_i z_i}{\sum_{i=1}^k \la_i}$. Then using
convexity of $\Psi(z, \cdot)$, we obtain
\begin{align}\label{ergodic}
e_r(Z^k) = \max_{z \in U} \Psi(z, Z^k) \leq
  \frac{1}{\sum_{i=1}^k{\la_i}} \max_{z\in U} \bigl(\sum_{i=1}^k
  \Psi(z,z^i)\bigr) \leq \frac{M}{ \sum_{i=1}^k{\la_i}}.
\end{align}
Taking into account that $(\la_k)$ is separated from zero, we obtain
the $O(1/k)$ convergence rate for the ergodic sequence $(Z^k)$.

\begin{remark}\label{rem:ergodic}
    For the case of the composite minimization problem~\eqref{comp},
    instead of using the merit function it is simpler to use the
    energy residual: $J(z^k) - J(z^*)$. For this we need to
    use in~\eqref{ls:eq:7} that
\begin{align}
\la_k \lr{F(z^k), z^k-z} + \la_k (g(z^{k})-g(z))  & \geq \la_k [f(z^k) -
    f(z) + g(z^k) - g(z)] \nonumber \\ & = \la_k (J(z^k) - J(z)).
\end{align}
In this way, we may proceed analogously to obtain
\begin{equation}
    \min_{i=1,\dots, k} (J(z^i)-J(z^*)) \leq \frac{M}{
    \sum_{i=1}^k{\la_i}}\quad \text{and}\quad J(Z^k) - J(z^*)\leq \frac{M}{ \sum_{i=1}^k{\la_i}}.
\end{equation}
\end{remark}

\subsection{Linear convergence}
\label{sec:lin-conv}

For many VI methods it is possible to derive a linear convergence rate
under some additional assumptions. The most general tool for that is
the use of error bounds. For a survey of error bounds, we refer the
reader to~\cite{pang1997error} and for their applications to the VI
algorithms to \cite{tseng1995linear,solodov2003convergence}.

Let us briefly recall the terminology associated with  linear convergence. Suppose that $(u^k)\subset \E$ is a sequence that converges to $u\in \E$. We say that convergence is $Q$-linear, if there is $q \in (0,1)$ such that $\n{u^{k+1} - u} < q \n{u^k - u}$ for all $k$ large enough. We also say that convergence is $R$-linear, if for all $k$  large enough, $\n{u^k - u}\leq \gamma_k$ and $(\gamma_k)$ converges $Q$-linearly to zero.

Let us fix some $\la >0$ and define the natural residual
$r(z, \la):= z - \prox_{\la g}(z - \la F(z))$. Evidently, $z\in S$ if
and only if $r(z, \la) = 0$. We say that problem~\eqref{vip} satisfies
an error bound condition if there exist positive constants $\mu$ and $\eta$ such that
\begin{equation}
    \label{eq:error} \dist(z, S) \leq \mu \n{r(z, \la)} \qquad \forall
    x \quad \text{with}\quad \n{r(z, \la)}\leq  \eta.
\end{equation}
The function $\la \mapsto \n{r(z, \la)}$ is nondecreasing and $\la \mapsto \frac{\n{r(z, \la)}}{\la}$is nonincreasing  (see
\cite[Proposition 10.3.6]{pang:03}), and thus all natural residuals $r(\cdot, \la)$ are
equivalent. Hence the choice of $\la$ in the above definition is not
essential. No doubt, it is not an easy task to decide whether
\eqref{eq:error} holds for a particular problem. Several examples are
known, see for instance \cite{solodov2003convergence,pang:03} and it
is still an important and active area of research.

In the analysis below we are not interested in sharp constants, but
rather in showing the linear convergence for $(z^k)$. This will allow
us to keep the presentation simpler. For the same reason we assume that
$F$ is $L$-Lipschitz continuous.

Choose any $\la > 0$ such that $\la_k \geq \la$ for all $k$. Without
loss of generality, we assume that $\la$ is the same as in
\eqref{eq:error}. As $\la \mapsto \n{r(\cdot, \la)}$ is
nondecreasing and $\prox_{\la_k g}$ is nonexpansive, using the
triangle inequality, we obtain
\begin{align} \label{eq:chain}
    \n{r(\z^k, \la)} & \leq \n{r(\z^k, \la_k)} = \n{\prox_{\la_k
    g}(\z^k-\la_k F(\z^k)) - \z^k} \nonumber \\ & \leq  \n{\prox_{\la_k
    g}(\z^k-\la_k F(\z^k)) -\prox_{\la_k
    g}(\z^k-\la_k F(z^k))} + \n{z^{k+1}-\z^k} \nonumber \\ & \leq \la_k L
    \n{z^k-\z^k} + \n{z^{k+1}-z^k}+\n{z^{k}-\z^k}\nonumber  \\ & = (1+\la_k
    L)\n{z^k-\z^k} +\n{z^{k+1}-z^k},
\end{align}
From this it follows that
\begin{equation}
    \label{eq:error-ineq0}
    \n{r(\z^k, \la)}^2 \leq 2(1+\la_kL)^2 \n{z^k-\z^k}^2 + 2\n{z^{k+1}-z^k}^2.
\end{equation}
Let $\beta = \frac{\phi}{\phi -1}$.  If $(\th_k)$ is separated from
zero, then the above inequality ensures that for any $z^k, \z^k$  and
$\e \in (0,1)$ there exists $m\in (0,1)$ such that
\begin{equation}
    \label{eq:error-ineq}
    m \beta \mu^2 \n{r(\z^k, \la)}^2 \leq \th_k \n{z^k-\z^k}^2 + \th_k
    \frac{\e}{2} \n{z^{k+1}-z^k}^2.
\end{equation}
The presence of so many constants in \eqref{eq:error-ineq} will be
clear later.  In order to proceed, we have to modify
Alg.~\ref{alg:egraal}. Now instead of \eqref{la}, we choose the stepsize
by
\begin{equation}    \label{la-linear}
    \la_k = \min\Bigl\{ \rho \la_{k-1},\ \frac{\phi \d \th_{k-1}
    }{4\la_{k-1}}\frac{\n{z^k-z^{k-1}}^2}{\n{F(z^k)-F(z^{k-1})}^2},\
  \bar \la\Bigr\}, \qquad \d \in (0,1).
\end{equation}
This modification basically means that we slightly bound the
stepsize. However, this is not crucial for the steps, as we can choose
$\d$ arbitrary close to one.  An argument completely analogous to that
in the proof of Lemma~\ref{l:separ_la1} (up to the factor $\d$) shows that
both $(\la_k)$ and $(\th_k)$ are bounded and separated from zero. This
confirms correctness of our arguments about $(\la_k)$ and $(\th_k)$ in
\eqref{eq:chain} and \eqref{eq:error-ineq0}.  It should be also obvious
that Alg.~\ref{alg:egraal} with \eqref{la-linear} instead of \eqref{la}
has the same convergence properties, since \eqref{ls:eq:10} --- the
only place where this modification plays some role --- will be still
valid.  For any $\d \in (0,1)$ there exist $\e\in (0,1)$ and
$m\in (0,1)$ such that $\d = (1-\e) (1-m)$ and \eqref{eq:error-ineq} is
fulfilled for any $z^k, \z^k$.  Now using \eqref{la-linear}, one can
derive a refined version of \eqref{ls:eq:10}:
\begin{align}
     2\la_k & \lr{F(z^k)-F(z^{k-1}),z^k-z^{k+1}}  \leq
     2\la_k\n{F(z^k)-F(z^{k-1})}\n{z^k-z^{k+1}} \nonumber \\ & \leq \sqrt{\d \th_k
     \th_{k-1}}\n{z^k-z^{k-1}}\n{z^k-z^{k+1}} \nonumber \\ & \leq
     \frac{(1-\e)\th_k}{2} \n{z^{k+1}-z^k}^2  +  \frac{(1-m)\th_{k-1}}{2} \n{z^{k}-z^{k-1}}^2.
 \end{align}
With this inequality, instead of \eqref{ls:eq:11}, for $z=z^*\in S$ we have
\begin{align}
    \label{ls:eq:11-linear} 
    \quad  & \b \n{\z^{k+1}-z^*}^2  +
    \frac{\th_k}{2} \n{z^{k+1}-z^k}^2 +  2\la_k \Psi(z^*, z^k)
     \nonumber \\ \leq \  & \b
    \n{\z^{k}-z^*}^2 +  \frac{\th_{k-1}(1-m)}{2}\n{z^k - z^{k-1}}^2 - \th_k
    \n{z^k-\z^k}^2 - \th_k \frac{\e}{2} \n{z^{k+1}-z^k}^2  \nonumber
  \\  \leq \ & \b
    \n{\z^{k}-z^*}^2 +  \frac{\th_{k-1}(1-m)}{2}\n{z^k - z^{k-1}}^2 - m \b \mu^2
    \n{r(\z^k, \la)}^2,
\end{align}
where in the last inequality we have used \eqref{eq:error-ineq}. As
$(\z^k)$ converges to a solution, $r(\z^k, \la)$ goes to $0$, and
hence $\n{r(\z^k, \la)} \leq \eta$ for all $k\geq k_0$. Setting
$z^*=P_S(\z^k)$ in \eqref{ls:eq:11-linear} and using~\eqref{eq:error}
and that $\Psi(z^*,z^k)\geq 0$, we obtain
\begin{align}
   \b \dist(\z^{k+1},S)^2  & +
    \frac{\th_k}{2} \n{z^{k+1}-z^k}^2  \leq
    \b \n{\z^{k+1}-z^*}^2  +
    \frac{\th_k}{2} \n{z^{k+1}-z^k}^2 \nonumber \\ & \leq  (1-m)\bigl(\b
    \dist(\z^k, S)^2 +  \frac{\th_{k-1}}{2} \n{z^k - z^{k-1}}^2\bigr).
\end{align}
From this the $Q$-linear rate of convergence for the sequence
$\bigl(\dist(\z^k, S)^2 + \frac{\th_{k-1}}{2}\n{z^k - z^{k-1}}^2\bigr)$
follows. Since $(\th_k)$ is separated from zero, we conclude that
$\n{z^{k}-z^{k-1}}$ converges $R$-linearly and this immediately implies
that the sequence $(z^k)$ converges $R$-linearly.  We summarize the
obtained result in the following statement.
\begin{theorem}
    Suppose that conditions
    \hyperref[C1]{(C1)--(C3)} are satisfied, $F\colon \dom g \to V$ is $L$--Lipschitz and the error bound~\eqref{eq:error} holds. Then $(z^k)$, generated by Alg.~\ref{alg:egraal} with \eqref{la-linear} instead of \eqref{la}, converges to a solution of \eqref{vip} at least $R$-linearly.
\end{theorem}
\section{Fixed point algorithms}
\label{sec:fixed}
Although in general it is a very standard way to formulate a VI
\eqref{vip} as a fixed point equation $x = \prox_g(x - F(x))$,
sometimes other way around might also be beneficial. In this section
we show how one can apply general algorithms for VI to find a fixed
point of some operator $T\colon \E \to \E$. Clearly, any fixed point
equation $x = Tx$ is equivalent to the equation $F(x) = 0$ with
$F=\id -T$. The latter problem is of course a particular instance of
\eqref{vip} with $g\equiv 0$. Hence, we can work under the assumptions
of Remark~\ref{rem:particular}.

By $\Fix T$ we denote the fixed point set of the operator
$T$. Although, with a slight abuse of notation, we will not use
brackets for the argument of $T$ (this is common in the fixed point
literature), but we continue doing that for the argument of $F$.

We are interested in the following classes of operators:
\begin{enumerate}[(a)]
    \item Firmly-nonexpansive:
    $$\n{Tx-Ty}^2\leq \n{x-y}^2 - \n{(x-Tx)-(y-Ty)}^2\quad \forall
    x,y\in \E.$$
    \item Nonexpansive:
    $$\n{Tx-Ty}\leq \n{x-y} \quad \forall x,y\in \E.$$

    \item Quasi-nonexpansive:
    $$\n{Tx-\x}\leq \n{x-\x} \quad \forall x\in \E, \forall \x\in \Fix T.$$

    \item Pseudo-contractive:
    $$\n{Tx-Ty}^2\leq \n{x-y}^2 + \n{(x-Tx)-(y-Ty)}^2\quad \forall
    x,y\in \E.$$

    \item Demi-contractive:
    $$
  \n{Tx-\x}^2\leq \n{x-\x}^2 + \n{x-Tx}^2\quad \forall     x\in \E,
  \forall \x\in \Fix T.
  $$
\end{enumerate}
We have the following obvious relations
\begin{equation}
\begin{cases}
    (a)\subset (b)\subset (c) \subset (e)\\
    (a)\subset (b)\subset (d) \subset (e).
\end{cases}
\end{equation}
Therefore, (e) is the most general class of the aforementioned
operators. Sometimes in the literature
\cite{qihou1987naimpally,berinde2007,naimpally1983extensions} the
authors consider operators that satisfy a more restrictive condition
\begin{equation}\label{rho} \n{Tx-\x}^2\leq \n{x-\x}^2 + \rho \n{x-Tx}^2\quad \forall x\in \E, \forall \x\in \Fix T \quad \rho \in [0,1]
\end{equation}
and call them $\rho$--demi-contractive for $\rho\in [0,1)$ and
hemi-contractive for $\rho = 1$. We consider only the most general
case with $\rho=1$, but still for simplicity will call them as
demi-contractive. It is also tempting to call the class (e) as
quasi-pseudocontractive by keeping the analogy between (b) and (c),
but this tends to be a bit confusing due to ``quasi-pseudo''. Notice
that for the case $\rho<1$ in \eqref{rho}, one can consider a relaxed
operator $S = \rho \id + (1-\rho)T$. It is easy to see that $S$
belongs to the class~(c) and $\Fix S = \Fix T$. However, with
$\rho = 1$, the situation becomes much more difficult.

When $T$ belongs to (a) or (b), the standard way to find a fixed point
of $T$ is by applying the celebrated Krasnoselskii-Mann scheme (see
e.g. \cite{baucomb}): $x^{k+1} = \a x^k + (1-\a)Tx^k$, where
$\a \in (0,1)$. The same method can be applied when $T$ is in
class (c), but to the averaged operator $S = \b \id + (1-\b)T$,
$\b \in (0,1)$, instead of $T$. However, things become more
difficult when we consider broader classes of operators.  In
particular, Ishikawa in \cite{ishikawa1974fixed} proposed an iterative
algorithm when $T$ is Lipschitz continuous and
pseudo-contractive. However, its convergence requires a  compactness
assumption, which is already too restrictive. Moreover, the
convergence is rather slow, since the scheme uses some auxiliary
slowly vanishing sequence (as in the subgradient method), also the
scheme uses two evaluations of $T$ per iteration. Later, this scheme
was extended in \cite{qihou1987naimpally} to the case when $T$ is
Lipschitz continuous and demi-contractive but with the same
assumptions as above.

Obviously, one can rewrite condition (e) as
\begin{equation}\label{fix:demi} \lr{Tx-x, \x-x}\geq 0 \quad \forall x\in \E,\, \forall \x \in \Fix T,
\end{equation} which means that the angle between vectors $Tx-x$ and $\x -x$ must always be nonobtuse.

We know that $T$ is pseudo-contractive if and only if $F$ is monotone~\cite[Example 20.8]{baucomb}. It is not more difficult to check
that $T$ is demi-contractive if and only if $F$ satisfies
\hyperref[C4]{(C4)}. In fact, in this case $S = \Fix T$, thus,
\eqref{fix:demi} and \hyperref[C4]{(C4)} are equivalent:
$\lr{F(x), x-\x} = \lr{x-Tx, x-\x} \geq 0$. The latter observation
allows one to obtain a simple way to find a fixed point of a
demi-contractive operator $T$.  In particular, in order to find a
fixed point of $T$, one can apply the aGRAAL. Moreover, since in our
case $g\equiv 0 $ and $\prox_g = \id$, \eqref{zk+1} simplifies to
\begin{align}
    \label{fix:zk+1} x^{k+1} = \x^k - \la_k F(x^k) & = (\x^k - \la_k
    x^k) + \la_k Tx^k \nonumber \\ & = \bigl(\frac{\phi -1}{\phi}-\la_k\bigr)x^k +
    \frac 1 \phi \x^{k-1} + \la_k Tx^k.
\end{align}
From the above it follows:
\begin{theorem}\label{th:fix} Let $T\colon \E \to \E$ be locally Lipschitz and demi-contractive operator. Define $F = \id - T$. Then the sequence $(x^k)$ defined by aGRAAL with \eqref{fix:zk+1} instead of \eqref{zk+1} converges to a fixed point of $T$.
\end{theorem}
\begin{proof} Since $F=\id -T$ is obviously locally Lipschitz, the proof is an immediate application of Theorem~\ref{th:egraal}. 
\end{proof}

\begin{remark}\label{r:km}
    The obtained theorem is interesting not only for the very general
    class of demi-contractive operators, but for a more limited class
    of non-expansive operators. The scheme \eqref{fix:zk+1} requires
    roughly the same amount of computations as the Krasnoselskii-Mann
    algorithm, but the former method defines $\la_k$ from the local
    properties of $T$, and hence, depending on the problem, it can be
    much larger than $1$.  Recently, there appeared some papers on
    speeding-up the Krasnoselskii-Mann (KM) scheme for nonexpansive
    operators \cite{themelis2016supermann,giselsson2016line}.  The
    first paper proposes a simple linesearch in order to accelerate
    KM. However, as it is common for all linesearch procedures, each
    inner iteration requires an evaluation of the operator $T$, which
    in the general case can eliminate all advantages of it. The second
    paper considers a more general framework, inspired by Newton
    methods. However, its convergence guarantees are more restrictive.
\end{remark}
The demi-contractive property can be useful when we want to analyze
the convergence of the iterative algorithm $x^{k+1} = Tx^k$ for some
operator $T$. It might happen that we cannot prove that $(x^k)$ is
Fej\'er monotone w.r.t.\ $\Fix T$, but instead we can only show that
$\n{x^{k+1}-\x}^2 \leq \n{x^{k}-\x}^2 + \n{x^{k+1}-x^k}^2$ for all
$\x \in \Fix T$. This estimation guarantees that $T$ is
demi-contractive and hence one can apply Theorem~\ref{th:fix} to obtain a
sequence that converges to $\Fix T$.

Finally, we can relax condition in (e) to the following
\begin{equation}\label{eq:e_relax}
    \n{Tx-P_Sx}^2\leq \n{x-P_Sx}^2 + \n{x-Tx}^2\quad \forall x\in \E,
\end{equation}
which in turn is equivalent to $\lr{F(x), x-P_Sx}\geq 0$ for all $x\in \E$. For instance, \eqref{eq:e_relax} might arise when we know that $F=\id - T$ satisfies a global error bound condition
\begin{equation}
    \label{eq:error2} \dist(x, S) \leq \mu \n{F(x)} = \mu \n{x-Tx}\qquad \forall x\in \E,
\end{equation}
and it holds that
\begin{equation}
    \label{eq:e_relax2} \n{Tx -\x}^2\leq (1+\e)\n{x-\x}^2 \qquad \forall x\in \E\,\, \forall \x \in S,
\end{equation}
where $\e>0$ is some constant. One can easily show that
\eqref{eq:e_relax} follows from \eqref{eq:error2} and
\eqref{eq:e_relax2}, whenever $\e < \frac{1}{\mu^2}$. The proof of
convergence for aGRAAL with this new condition will be almost the
same: in the $k$-th iteration instead of using arbitrary $x^*\in S$,
we have to set $x^* = P_Sx^k$. Of course, the global error bound
condition \eqref{eq:error2} is too restrictive and is difficult to
verify \cite[Chapter~6]{pang:03}. On the other hand, property~\eqref{eq:e_relax2}
is very attractive: it is often the case that for some iterative
scheme one can only show \eqref{eq:e_relax2}, which is not enough to
proceed further with standard arguments like the Krasnoselskii-Mann
theorem or Banach contraction principle. We believe it might be an
interesting avenue for further research to consider the above settings
with $\mu$ and $\e$ dependent on $x$ in order to eliminate the
limitation of \eqref{eq:error2}. Condition \eqref{eq:e_relax2} also
relates to the recent work \cite{luke2016quantitative} where the
generalization of nonexpansive operators for multivalued case was
considered.

\section{Generalizations}
\label{sec:general} This section deals with two
generalizations. First, we present aGRAAL in a general metric
settings. Although this extension should be straightforward for most
readers, we present it to make this work self-contained. Next, by
revisiting the proof of convergence for aGRAAL we relax the monotonicity condition to a new one and discuss its consequences.
\subsection{aGRAAL with different metrics}
\label{sub:gen_metric} Throughout section~\ref{sec:gold-ratio-algor} we have
been working in standard Euclidean metric $\lr{\cdot, \cdot}$ and
assumed that $F$ is monotone with respect to it. There are at least
two possible generalizations of how one can incorporate some metric
into Alg.~\ref{alg:egraal}. Firstly, the given operator $F$ may not be
monotone in metric $\lr{\cdot, \cdot}$, but it is so in
$\lr{\cdot, \cdot}_P$, induced by some symmetric positive definite
operator $P$. For example, the generalized proximal method
$z^{k+1} = (\id + P^{-1}G)^{-1}z^k $ for some monotone operator $G$
gives us the operator $T = (\id + P^{-1}G)^{-1}$ which is nonexpansive
in metric $\lr{\cdot,\cdot}_P$, and hence $F = \id -T$ is monotone in
that metric but is not so in $\lr{\cdot,\cdot}$. This is an important
example, as it incorporates many popular methods: ADMM,
Douglas-Rachford, PDHG, etc. Secondly, it is often desirable to
consider an auxiliary metric $\lr{\cdot,\cdot}_M$, induced by a symmetric
positive definite operator $M$, that will enforce faster
convergence. For instance, for saddle point problems, one may give
different weights for primal and dual variables, in this case one will
consider some diagonal scaling matrix $M$. The standard analysis of
aGRAAL (and this is common for other known methods) does not take into
account that $F$ is derived from a saddle point problem; it treats the
operator $F$ as a black box, and just use one stepsize $\la $ for both
primal and dual variables.  For example, the primal-dual hybrid
gradient algorithm~\cite{chambolle2011first}, which is a very popular
method for solving saddle point problems with a linear operator, uses
different steps $\t$ and $\s$ for primal and dual variables; and the
performance of this method drastically depends on the choice of these
constants. This should be kept in mind when one applies aGRAAL for
such problems.

Another possibility is if we already work in metric $\lr{\cdot,\cdot}_P$, then a good choice of the matrix $M$ can eliminate some undesirable computations, like computing the proximal operator in metric $\lr{\cdot, \cdot}_P$.  Of course, the idea to incorporate a specific metric to the VI algorithm for a faster convergence is not new and was considered, for example, in \cite{solodov1996modified,chen1997convergence,he1994new}. Our goal is to show that the framework of GRAAL can easily adjust to these new settings.

Let $M, P \colon \E \to \E$ be symmetric positive
definite operators. We consider two norms induces by $M$ and $P$
respectively
\begin{equation} \n{z}_{M} := \sqrt{\lr{z,z}_M} = \lr{Mz,z}^{1/2} \quad \text{and} \quad \n{z}_{P} = \sqrt{\lr{z,z}_P} = \lr{Pz,z}^{1/2}.
\end{equation}

For a symmetric positive definite operator $W$, we define the generalized proximal operator as $\prox_{g}^{W} = (\id + W^{-1} \partial g)^{-1}$.  Since now we work with the Euclidean metric $\lr{\cdot,\cdot}_P$, we have to consider a more general form of VI:
\begin{equation}
    \label{vip-gen} \text{find } \quad z^*\in \E \quad \text{ s.t. } \quad \lr{F(z^*), z-z^*}_P +g(z)-g(z^*) \geq 0 \quad \forall z\in \E,
\end{equation} where we assume that
\begin{enumerate}[ (D1) ]
    \label{D1}
    \item the solution set $S$ of \eqref{vip-gen} is nonempty.
    \item $g\colon \E \to (-\infty, +\infty]$ is a convex lsc function;
    \item $F\colon \dom g \to \E$ is $P$--monotone: $\lr{F(u)-F(v), u-v}_P\geq 0$ \quad $\forall u,v\in \dom g$.
\end{enumerate}

The modification of Alg.~\ref{alg:egraal} presented below assumes that the
matrices $M,P$ are already given. In this note we do not discuss how
to choose the matrix $M$, as it depends crucially on the problem
instance.

\begin{algorithm}[!ht]\caption{\textit{Adaptive Golden Ratio Algorithm for general metrics}}
    \label{alg:egraal-gen}
\begin{algorithmic}
    \STATE {\bfseries Input:} Choose $z^0, z^1 \in \E$, $\la_0>0$,
    $\phi \in (1,\varphi]$, $\bar \la >0$. Set $\z^0 = z^1$, $\th_0 = 1$,
    $\rho = \frac{1}{\phi} + \frac{1}{\phi^2}$.
    \STATE {\bfseries For $k\geq 1 $ do}
    \STATE \quad  1. Find the stepsize:
    \begin{equation} \label{fixed:la} \la_k = \min\Bigl\{ \rho
        \la_{k-1},\ \frac{\phi
        \th_{k-1}}{4\la_{k-1}}\frac{\n{z^k-z^{k-1}}^2_{MP}}{\n{F(z^k)-F(z^{k-1})}^2_{M^{-1}P}},\
        \bar \la\Bigr\}
\end{equation}
\STATE \quad  2. Compute the next iterates:
\begin{align}
  \z^k & = \frac{(\phi -1)z^k+ \z^{k-1}}{\phi}\\ z^{k+1}&
= \prox_{\la_k g}^{MP}(\z^{k}-\la_k M^{-1} F(z^k)).
\end{align}

\STATE \quad 3. Update: $\th_k = \phi \dfrac{\la_k}{\la_{k-1}}$.
\end{algorithmic}
\end{algorithm}

\begin{theorem}\label{th:egraal-gen} Suppose that $F\colon \dom g\to \E$ is locally
    Lipschitz continuous and conditions~\hyperref[D1]{(D1)--(D3)} are
    satisfied. Then $(z^k)$ and $(\z^k)$, generated by
    Alg.~\ref{alg:egraal-gen}, converge to a solution
    of~\eqref{vip-gen}.
\end{theorem}

\begin{proof} Fix any $z^*\in S$.  By the prox-inequality~\eqref{prox_ineq} we have
\begin{align}  \lr{M(z^{k+1}-\z^k)+ \la_k F(z^k), P(z^* - z^{k+1})} &\geq \la_k (g(z^{k+1})-g(z^*)) \label{fixed:eq:4}\\ \lr{M(z^{k}-\z^{k-1})+\la_{k-1} F(z^{k-1}), P(z^{k+1} - z^{k})} &\geq \la_{k-1} (g(z^{k})-g(z^{k+1})) \label{fixed:eq:5}.
\end{align} Multiplying \eqref{fixed:eq:5} by $\frac{\la_k}{\la_{k-1}} \geq 0$ and using that $\frac{\la_k}{\la_{k-1}} (z^k-\z^{k-1})=\th_k(z^k-\z^{k})$, we obtain
\begin{equation}
    \label{fixed:eq:6} \lr{\th_kM(z^{k}-\z^{k})+\la_k F(z^{k-1}), P(z^{k+1} - z^{k})} \geq \la_k (g(z^{k})-g(z^*))
\end{equation} Summation of~\eqref{fixed:eq:4} and~\eqref{fixed:eq:6} gives us
\begin{align}
    \label{fixed:eq:7}
     \lr{z^{k+1}-\z^k, z^* - z^{k+1}}_{MP} & + \th_k \lr{z^{k}-\z^{k},
     z^{k+1} - z^{k}}_{MP} + \la_k \lr{F(z^k)-F(z^{k-1}),
     z^k-z^{k+1}}_P \nonumber \\ & \geq \la_k (\lr{F(z^k), z^k-z^*}_P + g(z^k) -
     g(z^*))\geq 0.
\end{align}
 Expressing the  first two terms in~\eqref{fixed:eq:7}
through norms, we derive
\begin{multline}
    \label{fixed:eq:8} \n{z^{k+1}-z^*}^2_{MP} \leq \n{\z^k-z^*}^2_{MP}
    - \n{z^{k+1}-\z^k}^2_{MP}+2\la_k\lr{F(z^k)-F(z^{k-1}),z^k-z^{k+1}}_P \\+ \th_k \Bigl(\n{z^{k+1}-\z^k}^2_{MP}  - \n{z^{k+1}-z^k}^2_{MP}-\n{z^k-\z^k}^2_{MP}\Bigr).
\end{multline} 
Similarly to~\eqref{eq:identity}, we have
\begin{equation}\label{fixed:eq:identity}
  \n{z^{k+1}-z^*}^2_{MP} = \frac{\phi}{\phi-1}\n{\z^{k+1}-z^*}^2_{MP} - \frac{1}{\phi-1}\n{\z^{k}-z^*}^2_{MP}   +\frac{1}{\phi}\n{z^{k+1}-\z^k}^2_{MP}.
\end{equation}
Combining this with~\eqref{fixed:eq:8}, we derive
\begin{align}
  \label{fixed:eq:9} \frac{\phi}{\phi -1 } \n{\z^{k+1}-z^*}^2_{MP}
  \leq  &\frac{\phi}{\phi -1 } \n{\z^{k}-z^*}^2_{MP} + \bigl(\th_k - 1
          - \frac{1}{\phi}\bigr) \n{z^{k+1}-\z^{k}}^2\nonumber\\ &-\th_k
                                                                   \bigl(\n{z^{k+1}-z^k}^2_{MP}
                                                                   +
                                                                   \n{z^k-\z^k}^2_{MP}\bigr)
                                                                   \nonumber
  \\ &+ 2\la_k\lr{F(z^k)-F(z^{k-1}),z^k-z^{k+1}}_P.
\end{align} 
Notice that $\th_k\leq 1 + \frac 1 \phi$. Using \eqref{la-est2}, the rightmost term in~\eqref{fixed:eq:9} can be estimated as
\begin{align}
    \label{fixed:eq:10} 2\la_k\lr{F(z^k)-F(z^{k-1}),z^k-z^{k+1}}_P
    &\leq 2\la_k\n{F(z^k)-F(z^{k-1})}_{M^{-1}P} \n{z^k-z^{k+1}}_{MP}
  \nonumber \\
&\leq \sqrt{\th_k \th_{k-1}}\n{z^k-z^{k-1}}_{MP}
    \n{z^k-z^{k+1}}_{MP} \nonumber \\ &\leq \frac{\th_k}{2} \n{z^{k+1}-z^k}^2_{MP} + \frac{\th_{k-1}}{2} \n{z^{k}-z^{k-1}}^2_{MP}.
\end{align}
Applying the obtained estimation to~\eqref{fixed:eq:9}, we obtain
\begin{multline}
    \label{fixed:eq:11} \frac{\phi}{\phi -1 } \n{\z^{k+1}-z^*}^2_{MP} + \frac{\th_k}{2}\n{z^{k+1}-z^k}^2_{MP} \\ \leq \frac{\phi}{\phi -1 } \n{\z^{k}-z^*}^2_{MP} + \frac{\th_{k-1}}{2}\n{z^{k}-z^{k-1}}^2_{MP} -\th_k \n{z^k-\z^k}^2_{MP}.
\end{multline} It is obvious to see that the statement of
Lemma~\ref{l:separ_la1} is still valid for Alg.~\ref{alg:egraal-gen},
hence one can finish the proof by the same arguments as in the end of
Theorem~\ref{th:egraal}.  
\end{proof}

\subsection{Beyond monotonicity} \label{sec:beyond} 
A more careful examination of the proof of
Theorem~\ref{th:egraal} can help us to relax assumption
\hyperref[C3]{(C3)} (or \hyperref[C4]{(C4)}) even more. In particular,
one can impose
\begin{equation} \label{mvi}
 \exists \z\in \E \quad \text{s.t.}\quad \lr{F(z), z-\z} + g(z) - g(\z)\geq 0 \qquad \forall z\in \E.
\end{equation}
The above problem is known as Minty variational inequality (MVI)
associated with VI \eqref{vip}. Let $S_{MVI}, S_{VI}$ denote the
solution sets of MVI \eqref{mvi} and VI \eqref{vip}
respectively. Essentially, condition~\eqref{mvi} asks for
$S_{MVI}\neq \varnothing $.  It is a standard fact that when $F$ is
monotone, both problems are equivalent, see \cite[Lemma
1.5]{kinderlehrer1980introduction}. In general case ($F$ is
continuous), one can only claim that $S_{MVI}\subset S_{VI}$.
\begin{theorem}\label{th:be1} Suppose that $F\colon \dom g\to \E$ is
    locally Lipschitz continuous, $g\colon \E \to (-\infty, +\infty]$
    is convex lsc, and $S_{MVI}\neq \varnothing$. Then all cluster
    points of $(z^k)$, generated by Alg.~\ref{alg:egraal}, are
    solutions of~\eqref{vip}.
\end{theorem}
\begin{proof}
Fix any $\z \in S_{\text{MVI}}$. Then in the proof of
Theorem~\ref{th:egraal} instead of taking arbitrary $z$, choose $z = \z$. Then instead of \eqref{ls:eq:7}, we obtain
\begin{multline}
    \label{be:eq:7} \lr{z^{k+1}-\z^k, \z - z^{k+1}} + \th_k
    \lr{z^{k}-\z^{k}, z^{k+1} - z^{k}} + \la_k \lr{F(z^k)-F(z^{k-1}),
        z^k-z^{k+1}}\\ \geq \la_k [\lr{F(z^k), z^k - \z} +g(z^k) - g(\z)] \geq 0,
\end{multline} where the last inequality holds because of \eqref{mvi}. Proceeding as in \eqref{ls:eq:8}--\eqref{ls:eq:12}, we can deduce
\begin{multline}
    \label{be:eq:12} \frac{\phi}{\phi -1 } \n{\z^{k+1}-\z}^2 + \frac{\th_k}{2}\n{z^{k+1}-z^k}^2 + \sum_{i=2}^k\th_i \n{z^i-\z^{i}}^2 \\ \leq \frac{\phi}{\phi -1 } \n{\z^{2}-\z}^2 + \frac{\th_{1}}{2}\n{z^{2}-z^{1}}^2 -\th_2 \n{z^2-\z^2}^2.
\end{multline} From this we obtain that $(\z^k)$ is bounded and
$\th_k\n{z^k-\z^k}\to 0$. Then in the same way as in
Theorem~\ref{th:egraal}, one can show that all cluster points of
$(z^k)$ are elements of $S_{VI}$.  
\end{proof}
Of course, we obtain slightly weaker convergence guarantees than in
Theorem~\ref{th:egraal}, but on the other hand our assumptions are
much more general.

\section{Numerical experiments}
\label{sec:numer-exper}

This section collects several numerical experiments\footnote{All codes can be found on
\url{https://gitlab.gwdg.de/malitskyi/graal.git}.} to confirm our
findings.  Computations were performed
using Python 3.6 on a standard laptop running 64-bit Debian GNU/Linux.
In all experiments we take $\phi = 1.5$ for Alg.~\ref{alg:egraal}.

\subsection{Nash--Cournot equilibrium}
\label{sec:nash-equilibrium}
Here we study a Nash--Cournot oligopolistic equilibrium model. We give
only a short description, for more details we refer to
\cite{pang:03}. There are $n$ firms, each of them supplies a
homogeneous product in a non-cooperative fashion. Let $q_i\geq 0$
denote the $i$th firm's supply at cost $f_i(q_i)$ and
$Q = \sum_{i=1}^n q_i$ be the total supply in the market. Let $p(Q)$
denote the inverse demand curve. A variational inequality that
corresponds to the equilibrium is
\begin{equation}\label{nash}
\text{find}\quad q^* = (q_1^*,\dots, q_n^*)\in \R^{n}_+ \quad
\text{s.t. } \lr{F(q^*), q-  q^*}\geq 0,\quad
\forall q \in \R^{n}_+,
\end{equation}
where $F(q^*) = (F_1(q^*),\dots, F_n(q^*))$ and 
\begin{equation}
 F_i(q^*) = f'_i(q_i^*) - p\Bigl(\sum_{j=1}^n q_j^*\Bigr) - q_i^*
p'\Bigl(\sum_{j=1}^n q_j^*\Bigr)
\end{equation}

As a particular example, we assume that the inverse demand function $p$
and the cost function $f_i$ take the form:
\begin{equation}
 p(Q) = 5000^{1/\gamma}Q^{-1/\gamma} \qquad \text{and} \qquad f_i(q_i) = c_i q_i + \frac{\b_i}{\b_i+1}L_i^{\frac{1}{\b_i}} q_i^{\frac{\b_i+1}{\b_i}}
\end{equation}
with some constants that will be defined later. This is a classical
example of the Nash-Cournot equilibrium first proposed in
\cite{murphy1982mathematical} for $n=5$ players and later reformulated
as monotone VI in \cite{harker1984variational}. To make the problem
even more challenging, we set $n=1000$ and generate our data randomly
by two scenarios. Each entry of $\b$, $c$ and $L$ are drawn
independently from the uniform distributions with the following
parameters:
\begin{enumerate}[(a)]
    \item $\gamma =1.1$, $\b_i\sim \mathcal{U}(0.5,2)$,
    $c_i\sim \mathcal{U}(1,100)$, $L_i\sim \mathcal{U}(0.5,5)$;

    \item $\gamma =1.5$, $\b_i\sim \mathcal{U}(0.3,4)$ and $c_i$, $L_i$ as above.
\end{enumerate}
Clearly, parameters $\b$ and $\gamma$ are the most important as they
control the level of smoothness of $f_i$ and $p$. There are two main
issues that make many existing algorithms not-applicable to this
problem. The first is, of course, that due to the choice of $\b$ and
$\gamma$, $F$ is not Lipschitz continuous. The second is that $F$ is
defined only on $\R^n_+$.  This issue, which was already noticed in
\cite{solodov1999new} for the case $n=10$, makes the above problem
difficult for those algorithms that compute $F$ at the point which is
a linear combination of the feasible iterates. For example, the
reflected projected gradient method evaluates $F$ at $2x^k-x^{k-1}$,
which might not belong to the domain of $F$.

For each scenario above we generate $10$ random instances and for
comparison we use the residual $\n{q - P_{\R^n_+}(q - F(q))}$, which
we compute in every iteration. The starting point is
$z^1 = (1,\dots, 1)$. We compare aGRAAL with Tseng's FBF method with
linesearch \cite{tseng00}.  The results are reported in
Fig.~\ref{fig:nash}. One can see that aGRAAL substantially outperforms the
FBF method. Note that the latter method even without linesearch requires
two evaluations of $F$, so in terms of the CPU time that distinction
would be even more significant.
\begin{figure}[t]
  \begin{subfigure}[b]{0.49\linewidth}
    \centering
   \includegraphics[width=\linewidth]{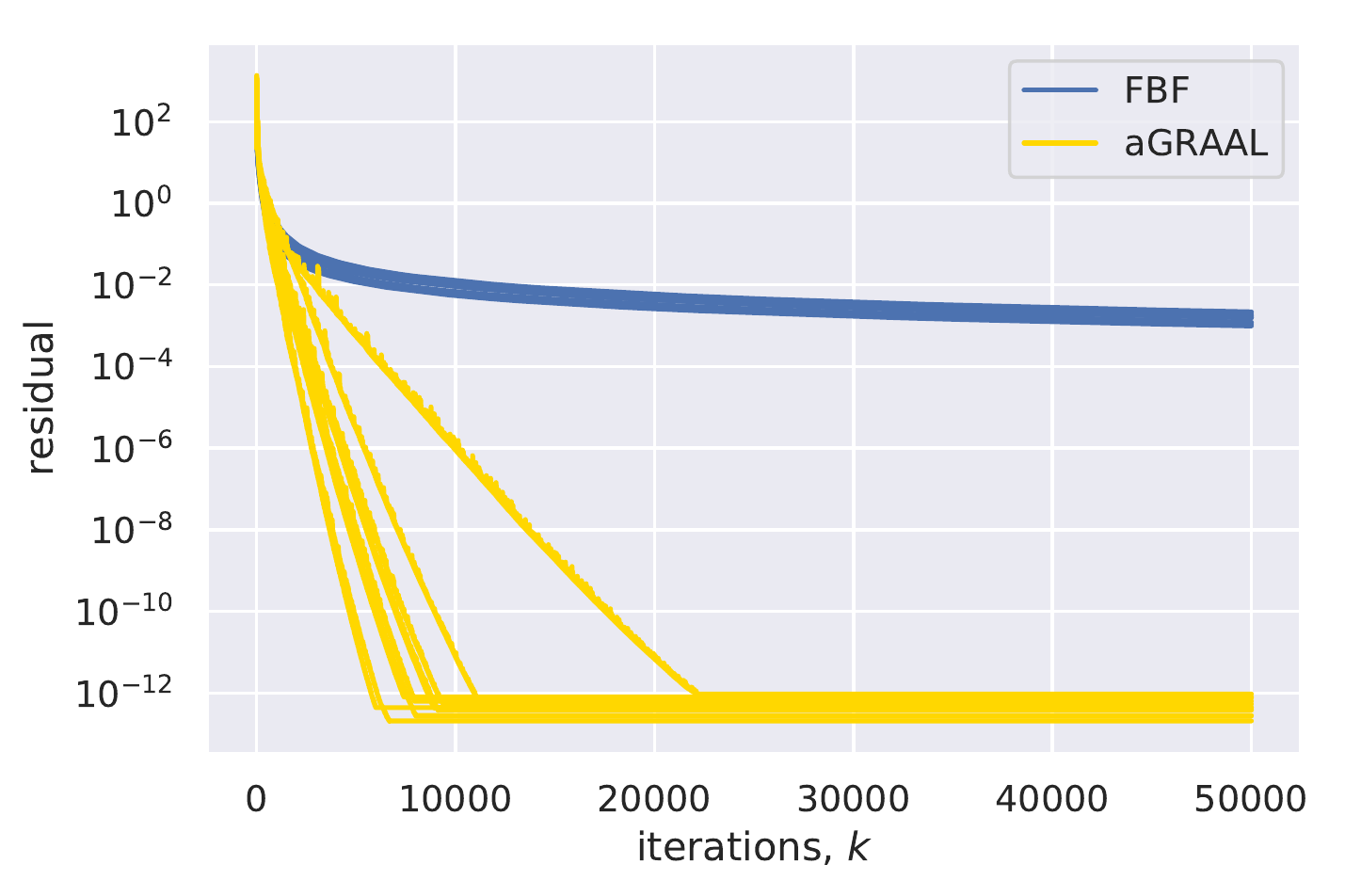}
  \end{subfigure}
  \begin{subfigure}[b]{0.49\linewidth}
    \centering
   \includegraphics[width=\linewidth]{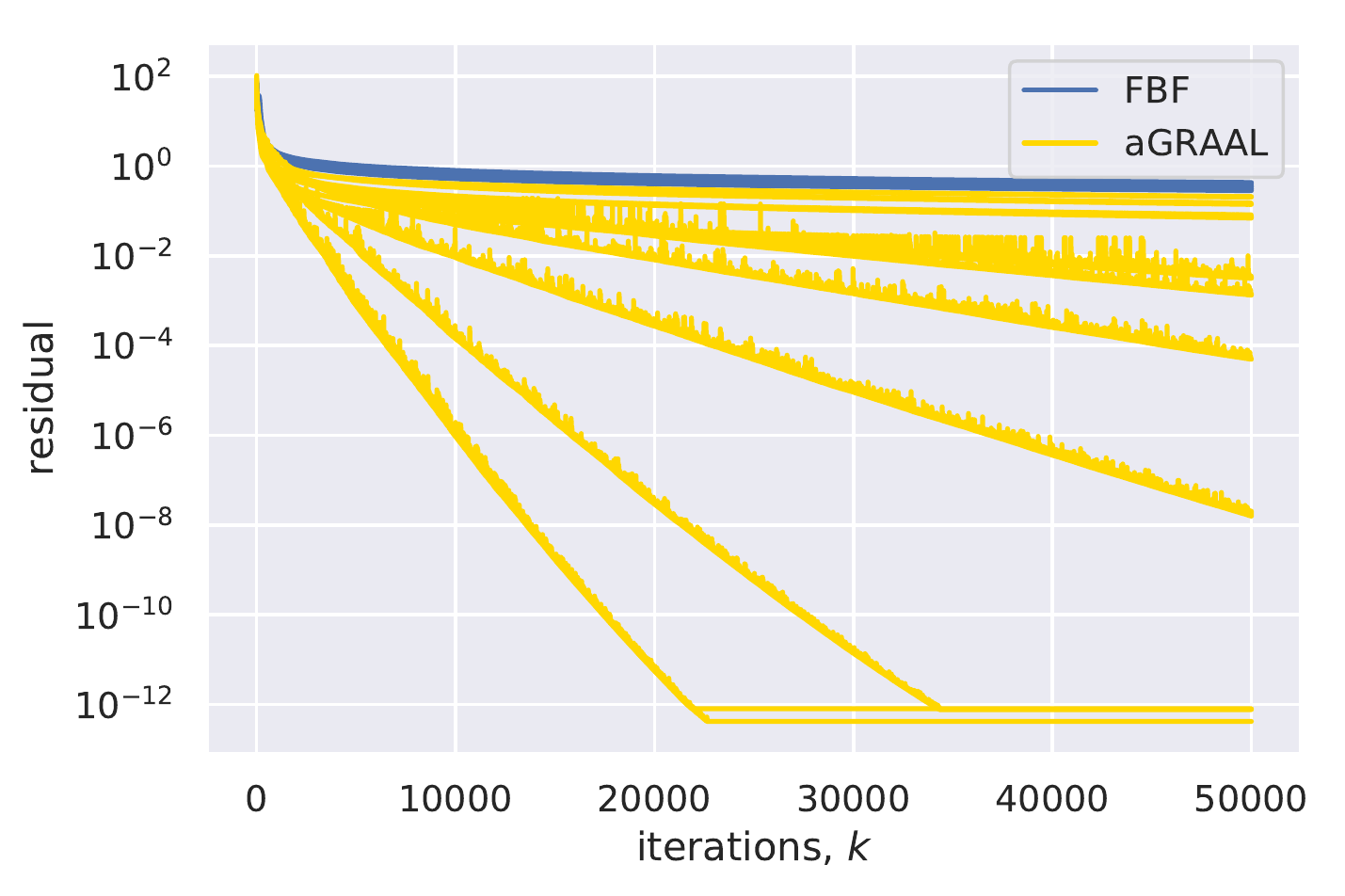}
\end{subfigure}
\caption{ Results for problem \eqref{nash}. Scenario (a) on the
left, (b) on the right.}
\label{fig:nash}
\end{figure}

\subsection{Convex feasibility problem}
Given a number of closed convex sets $C_i\subset \E$, $i=1,\dots,m$, the
convex feasibility problem (CFP) aims to find a point in their
intersection: $x \in \cap_{i=1}^m C_i$. The problem is very general
and allows one to represent many practical problems in this form.
Projection methods are a standard tool to solve such problem (we refer
to \cite{bauschke1996projection,combettes1996convex} for a more
in-depth overview). In this note we study a simultaneous projection
method: $x^{k+1} = Tx^k$, where
$T = \frac 1 m (P_{C_1}+\dots + P_{C_m})$. Its main advantage is that
it can be easily implemented on a parallel computing
architecture. Although it might be slower than the cyclic projection
method $x^{k+1}=P_{C_1}\dots P_{C_m}x^k$ in terms of iterations, for
large-scale problems it is often much faster in practice due to
parallelization and more efficient ways of computing $Tx$.

One can look at the iteration $x^{k+1} = Tx^k$ as an application of the
Krasnoselskii-Mann scheme for the firmly-nonexpansive operator
$T$. By that $(x^k)$ converges to a fixed point of $T$, which is either
a solution of CFP (consistent case) or a
solution of the problem $\min_x \sum_{i=1}^m \dist(x,C_i)^2$ (inconsistent
case when the intersection is empty).

To illustrate Remark~\ref{r:km}, we show how in many cases aGRAAL with
$F = \id -T$ can accelerate convergence of the simultaneous projection
algorithm. We believe, this is quite interesting, especially if one
takes into account that our framework works as a black box: it does
not require any tuning or a priori information about the initial
problem.

\textbf{Tomography reconstruction.}
The goal of the tomography reconstruction problem is to obtain a slice
image of an object from a set of projections (sinogram).
Mathematically speaking, this is an instance of a linear inverse
problem
\begin{equation}
    \label{eq:tomo}
    Ax = \hat b,
\end{equation}
where $x\in \R^{n}$ is the unknown image, $A\in \R^{m\times n}$ is the
projection matrix, and $\hat b\in \R^{m}$ is the given sinogram. In
practice, however, $\hat b $ is contaminated by some noise
$\e \in \R^m$, so we observe only $b = \hat b + \e$. It is clear that
we can formulate the given problem as CFP with
$ C_i = \{x\colon \lr{a_i,x} = b_i\}$.  However, since the projection
matrix $A$ is often rank-deficient, it is very likely that
$b\notin \range(A)$, thus we have to consider the inconsistent case
$\min_x \sum_{i=1}^m \dist(x,C_i)^2$.  As the projection onto $C_i$ is
given by $P_{C_i}x = x - \frac{\lr{a_i,x}-b_i}{\n{a_i}^2}a_i$,
computing $Tx$ reduces to the matrix-vector multiplications which is
realized efficiently in most computer processors. Note that our
approach only exploits feasibility constraints, which is definitely
not a state of the art model for tomography reconstruction. More
involved methods would solve this problem with the use of some
regularization techniques, but we keep such simple model for
illustration purposes only.

As a particular problem, we wish to reconstruct the Shepp-Logan
phantom image $256\times 256$ (thus, $x\in \R^{n}$ with $n=2^{16}$)
from the far less measurements $m = 2^{15}$. We generate the matrix
$A \in \R^{m\times n}$ from the \texttt{scikit-learn} library and
define $b = Ax + \e$, where $\e \in \R^m$ is a random vector, whose
entries are drawn from $\mathcal{N}(0,1)$. The starting point was
chosen as $x^1 = (0,\dots,0)$ and $\la_0 = 1$. In Fig.~\ref{fig:tomo}
(left) we report how the residual $\n{x^k-Tx^k}$ is changing w.r.t.\
the number of iterations and in Fig.~\ref{fig:tomo} (right) we show the
size of steps for aGRAAL scheme. Recall that the CPU time of both is
almost the same, so one can reliably state that in this case aGRAAL in
fact accelerates the fixed point iteration $x^{k+1}=Tx^k$.  The
information about the steps $\la_k$ gives us at least some explanation
of what we observe: with larger $\la_k$ the role of $Tx^k$ in
\eqref{fix:zk+1} increases and hence we are faster approaching to the
fixed point set.
\begin{figure}[t]
  \begin{subfigure}[b]{0.49\linewidth}
    \centering
   \includegraphics[width=\linewidth]{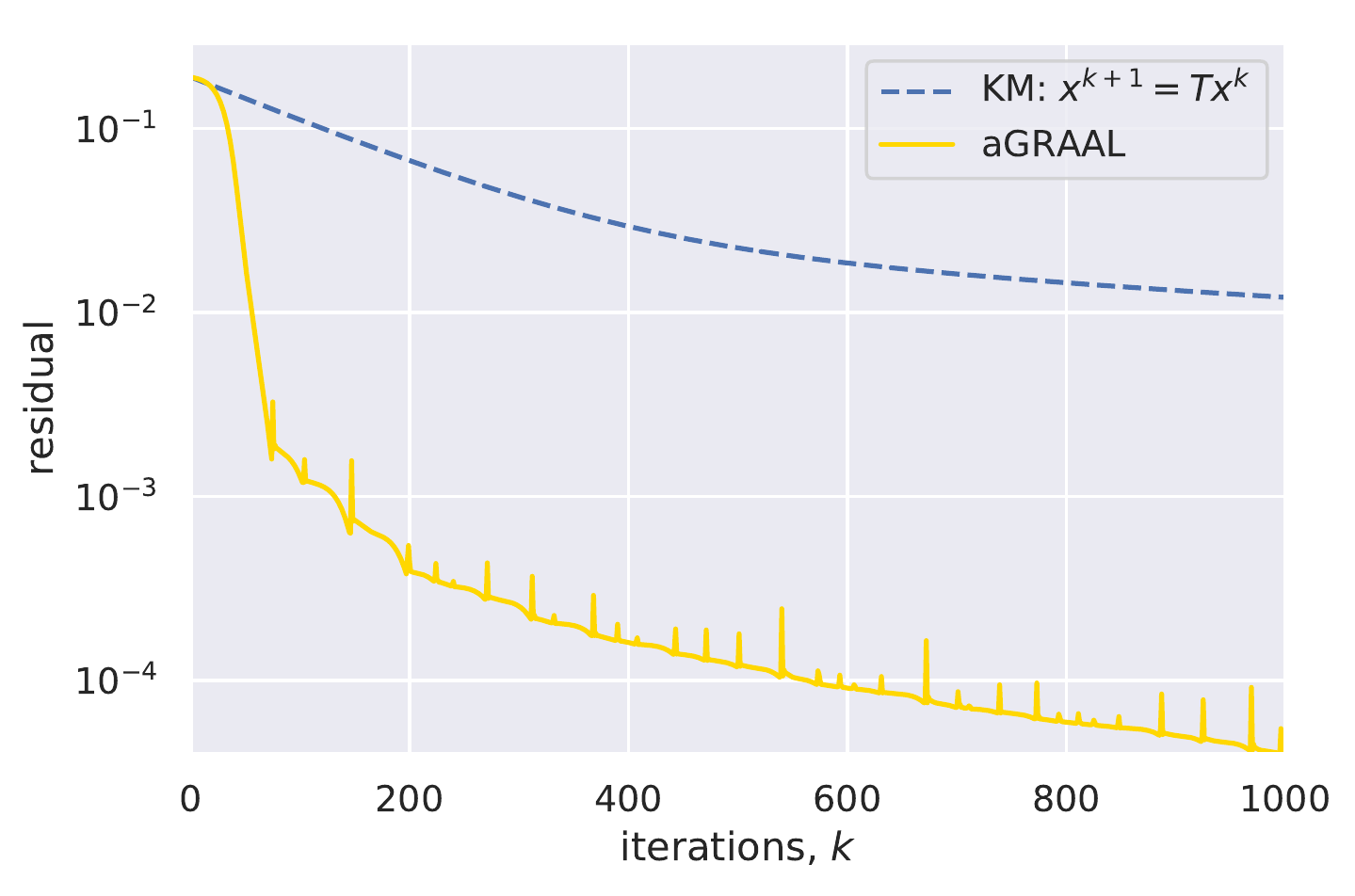}
  \end{subfigure}
  \begin{subfigure}[b]{0.49\linewidth}
    \centering
   \includegraphics[width=\linewidth]{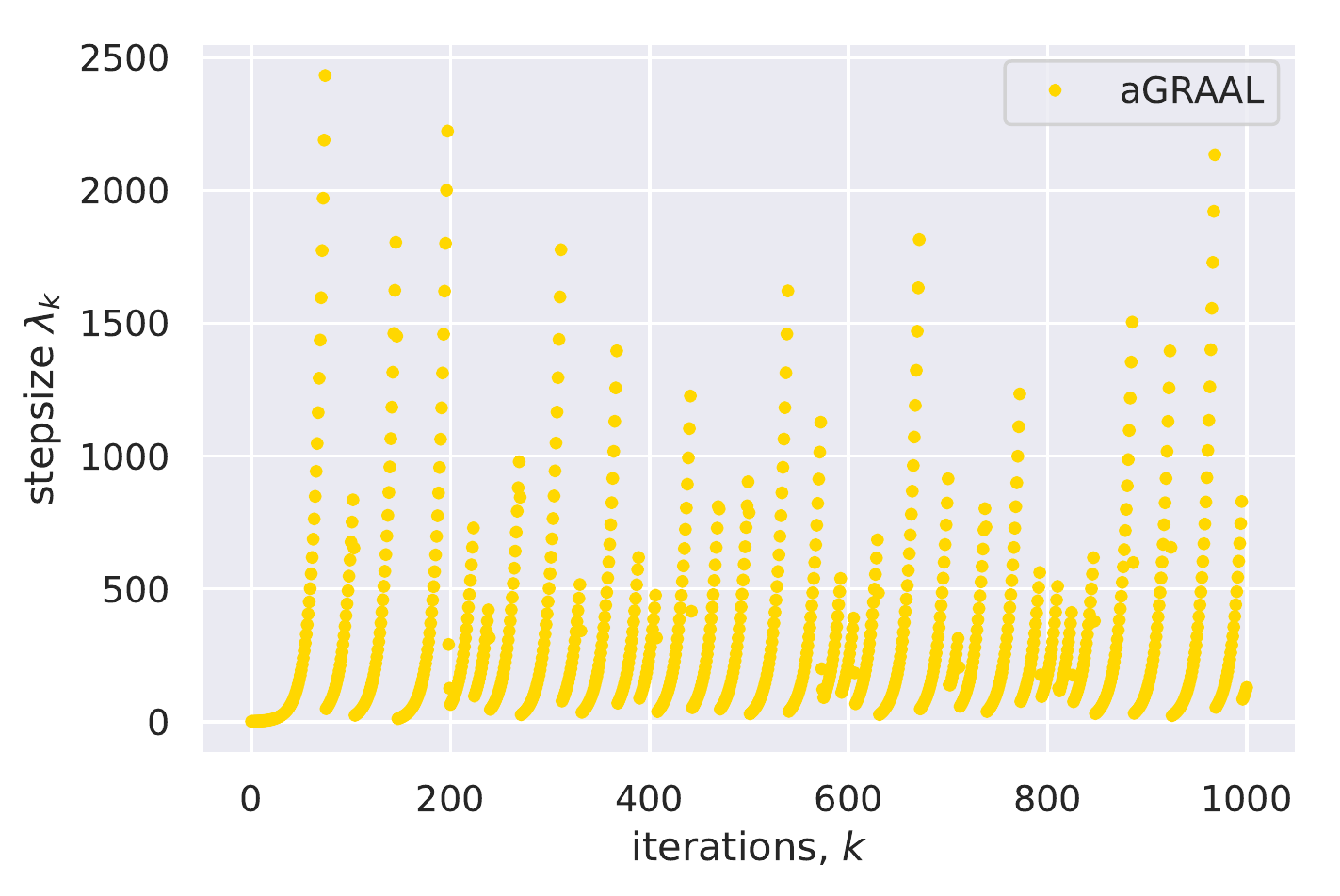}
\end{subfigure}
\caption{ Results for problem \eqref{eq:tomo}. Left: the behavior of residual $\n{x^k-Tx^k}$ for KM and aGRAAL
methods. Right: the size of stepsize of $\la_k$ for aGRAAL}
\label{fig:tomo}
\end{figure}

\textbf{Intersection of balls.}
Now we consider a synthetic nonlinear feasibility problem. We have to find a point
in $x\in \cap_{i=1}^mC_i$, where $C_i = \mathbb{B}(c_i, r_i)$, a closed
ball with a center $c_i\in \R^n$ and a radius $r_i>0$.
The projection onto $C_i$ is simple: $P_{C_i}x$ equals to
$\frac{x-c_i}{\n{x-c_i}}r_i$ if $\n{x-c_i}>r_i$ and $x$ otherwise.
Thus, again computing $Tx = \frac 1 m \sum_{i=1}^mP_{C_i}x$ can be done in parallel very efficiently.

We run two scenarios: with $n=1000$, $m=2000$ and with $n=2000$,
$m=1000$. Each coordinate of $c_i\in \R^n$ is drawn from
$\mathcal{N}(0,100)$. Then we set $r_i = \n{c_i}+1$ that ensures that
zero belongs to the intersection of $C_i$. The starting point was
chosen as the average of all centers:
$x^1 = \frac 1 m \sum_{i=1}^m c_i$. As before, since the cost of
iteration of both methods is approximately the same, we show only how
the residual $\n{Tx^k-x^k}$ is changing w.r.t.\ the number of
iterations. To eliminate the role of chance, we plot the results for
$100$ random realizations from each of the above
scenarios. Fig.~\ref{fig:balls} depicts the results. As one can see, the
difference is again significant.


\begin{figure}[ht]
  \begin{subfigure}[b]{0.5\linewidth}
    \centering
   \includegraphics[width=\linewidth]{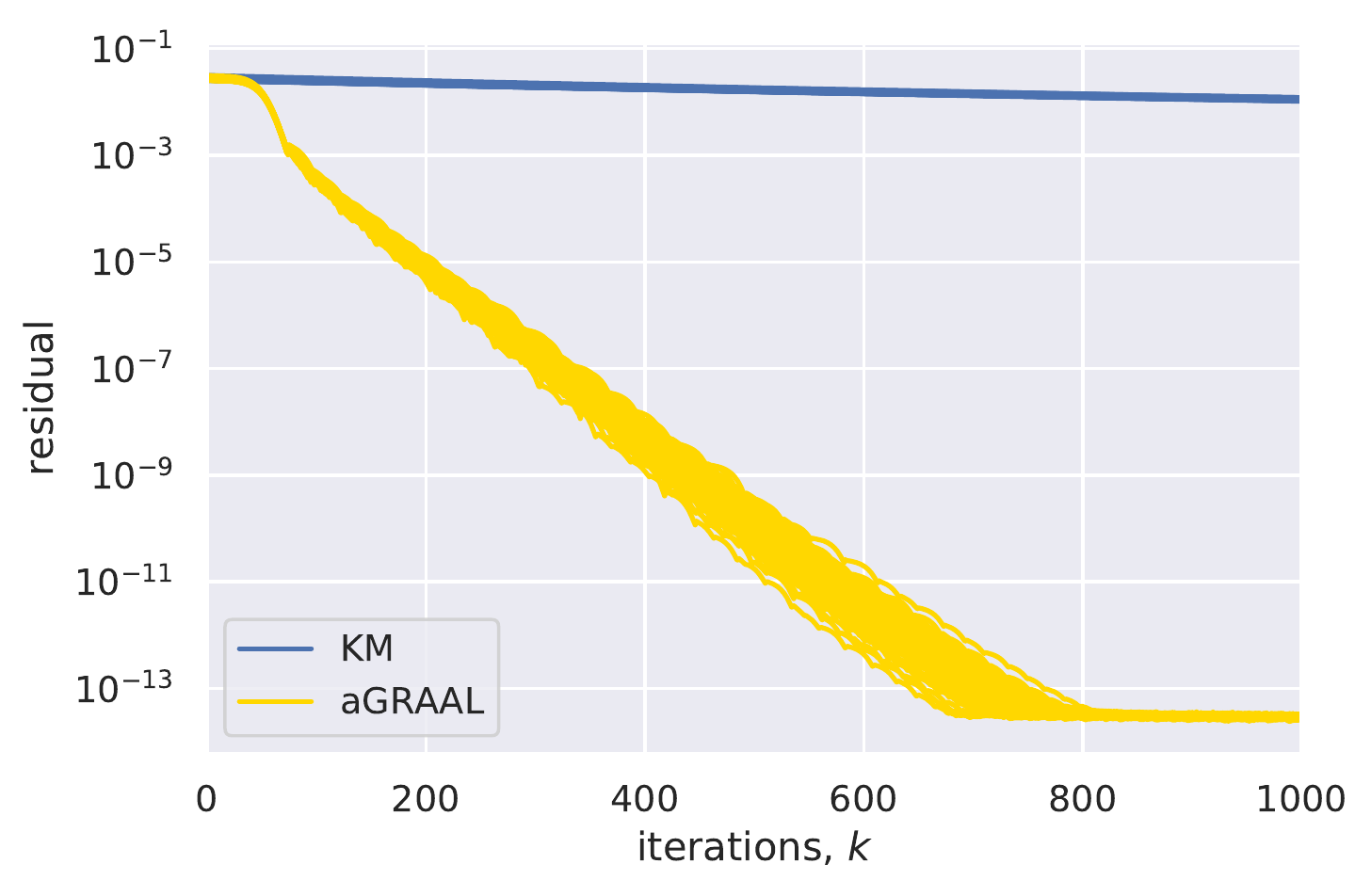}
  \end{subfigure}
  \begin{subfigure}[b]{0.5\linewidth}
    \centering
   \includegraphics[width=\linewidth]{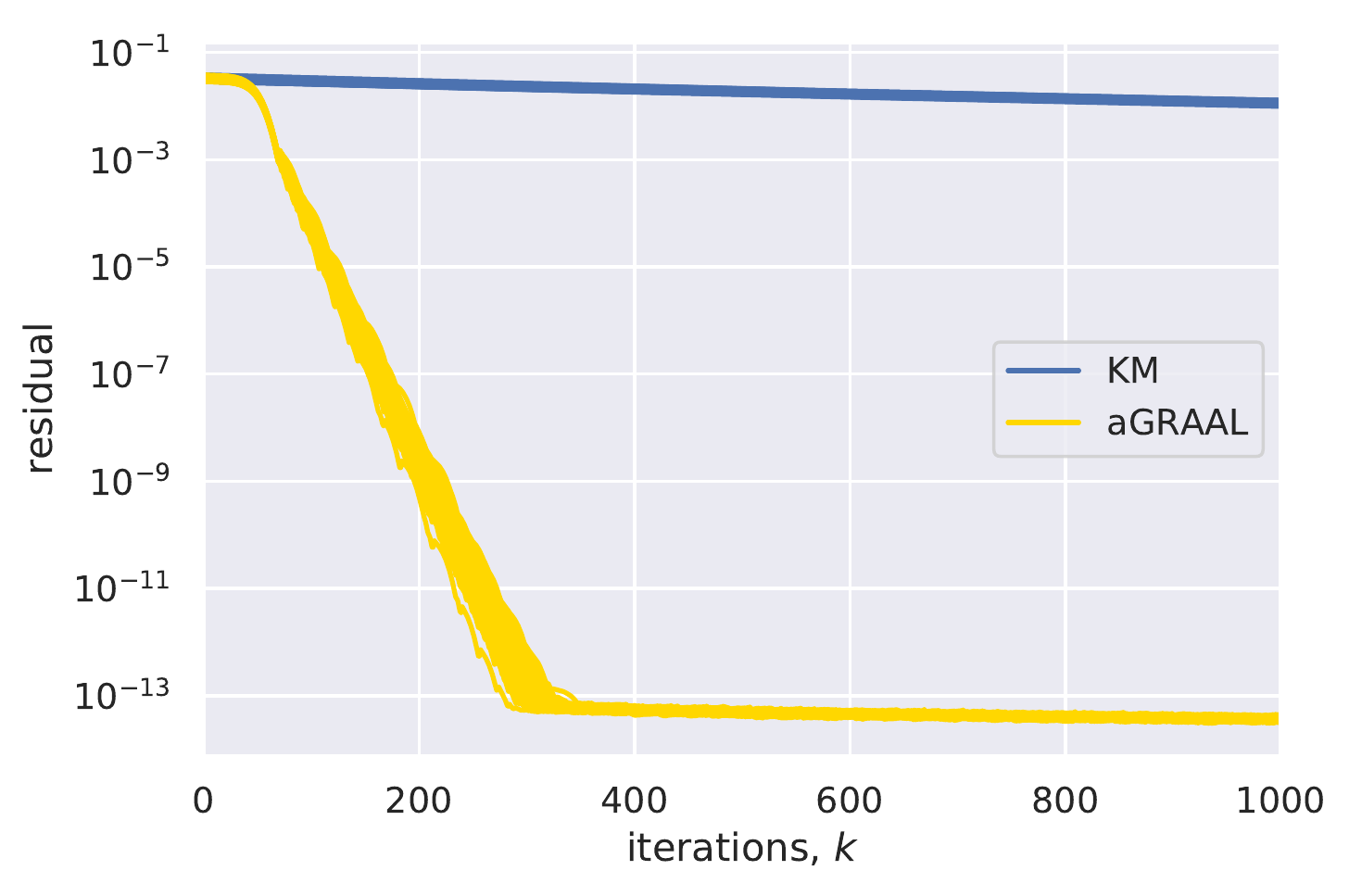}
\end{subfigure}
\caption{Convergence results for CFP with random balls $C_i$. 
$T = \frac{1}{m} \sum_{i=1}^m P_{C_i}$.  Left: $n=1000$,
$m=2000$. Right: $n=2000$, $m=1000$.}
\label{fig:balls}
\end{figure}

\subsection{Sparse logistic regression}
In this section we demonstrate that even for nice problems as convex
composite minimization~\eqref{comp} with Lipschitz $\nabla f$, aGRAAL
can be faster than the proximal gradient method or accelerated
proximal gradient method. This is interesting, especially since the
latter method has a better theoretical convergence rate.

Our problem of interest is a sparse logistic regression
\begin{equation}
    \label{eq:slr}
     \min_x J(x):=\sum_{i=1}^m \log(1+\exp (-b_i\lr{a_i, x})) + \gamma \n{x}_1,
\end{equation}
where $x\in \R^n$, $a_i\in \R^n$, and $b_i\in \{-1, 1\}$, $\gamma
>0$. This is a popular problem in machine learning applications where
one aims to find a linear classifier for points $a_i$.  Let us form a
new matrix $K\in \R^{m\times n}$ as $K_{ij} = -b_ia_{ij}$ and set
$\tilde f(y) = \sum_{i=1}^m\log(1+\exp (y_i))$. Then the objective in
\eqref{eq:slr} is $J(x) = f(x)+g(x)$ with $g(x) = \gamma \n{x}_1$ and
$f(x) = \tilde f(Kx)$. As $\tilde f$ is separable, it is easy to
derive that $L_{\nabla \tilde f} = \frac 1 4$. Thus,
$L_{\nabla f} = \frac 1 4 \n{K^TK}$.

We compare our method with the proximal gradient method (PGM) and
FISTA~\cite{fista} with a fixed stepsize.  We have not included into
comparison the extensions of these methods with linesearch, as we are
interested in methods that require one evaluation of $\nabla f$ and
one of $\prox_g$ per iteration. For both methods we set the stepsize
as $\la = \frac{1}{L_{\nabla f}} = \frac{4}{\n{K^TK}}$. We take two
popular datasets from LIBSVM~\cite{libsvm}: \emph{real-sim} with
$m=72309$, $n=20958$ and \emph{kdda-2010} with $m=510302$,
$n=2014669$. For both problems we set
$\gamma = 0.005\n{A^Tb}_{\infty}$. We are aware of that neither PGM
nor FISTA can be considered as the state of the art for
\eqref{eq:slr}, stochastic methods seem to be more competitive as the
size of the problem is quite large. Overall our motivation is not to
propose the best method for \eqref{eq:slr} but to demonstrate the
performance of aGRAAL on some real-world problems.

We run all methods for sufficiently many iterations and compute the
energy $J(x^k)$ in each iteration. If after $k$ iterations the
residual was small enough:
$\n{x^k-\prox_g(x^k - \nabla f(x^k))}\leq 10^{-6}$, we choose the
smallest energy value among all methods and set it to $J_*$. In
Fig.~\ref{fig:regr} we show how the energy residual $J(x^k)-J_*$ is
changing w.r.t.\ the iterations. Since the dimensions in both
problems are quite large, the CPU time for all methods is
approximately the same.
\begin{figure}[ht]
  \begin{subfigure}[b]{0.5\linewidth}
      \centering
      \includegraphics[width=\linewidth]{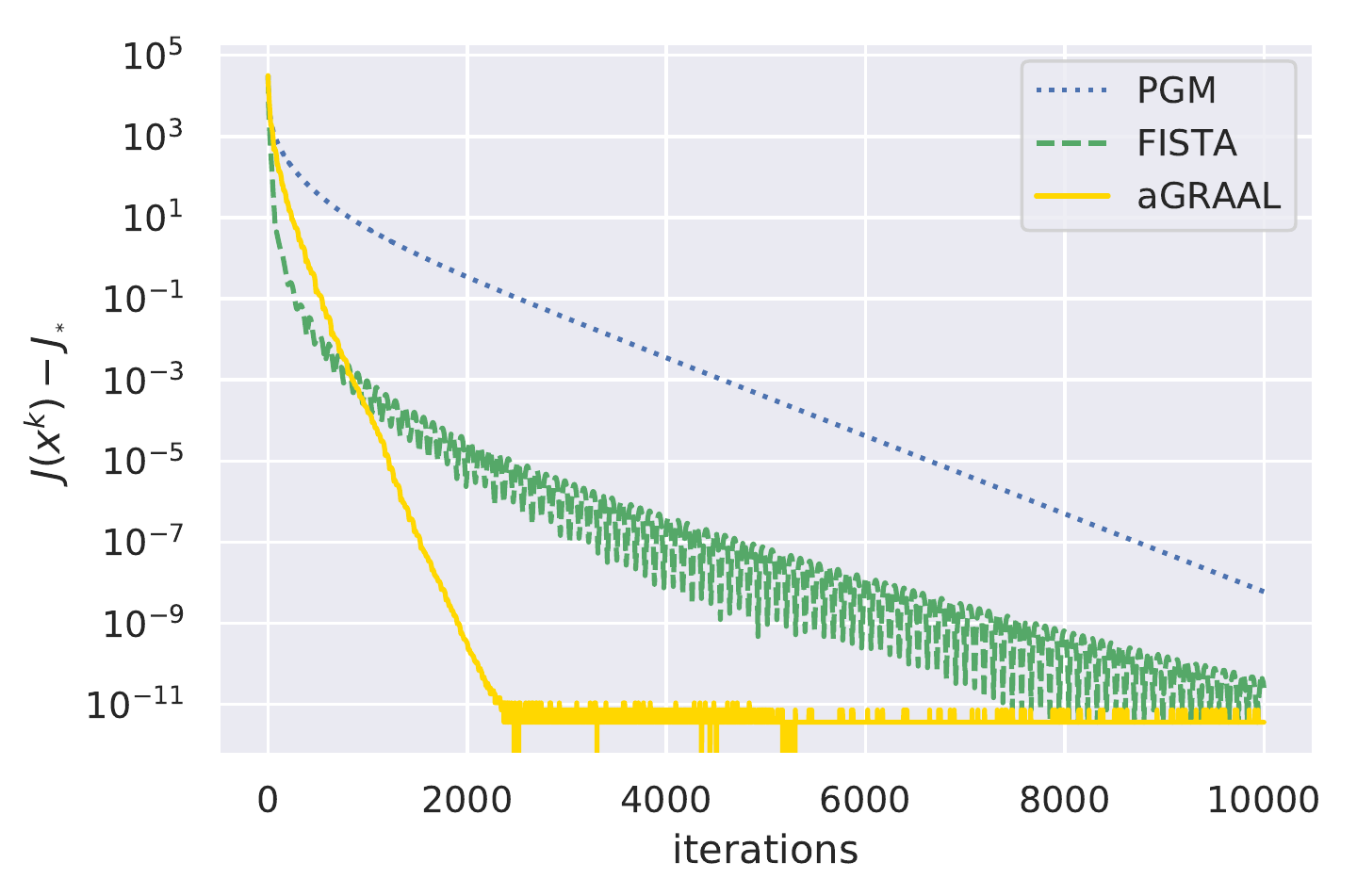}
  \end{subfigure}
  \begin{subfigure}[b]{0.5\linewidth}
    \centering  
    \includegraphics[width=\linewidth]{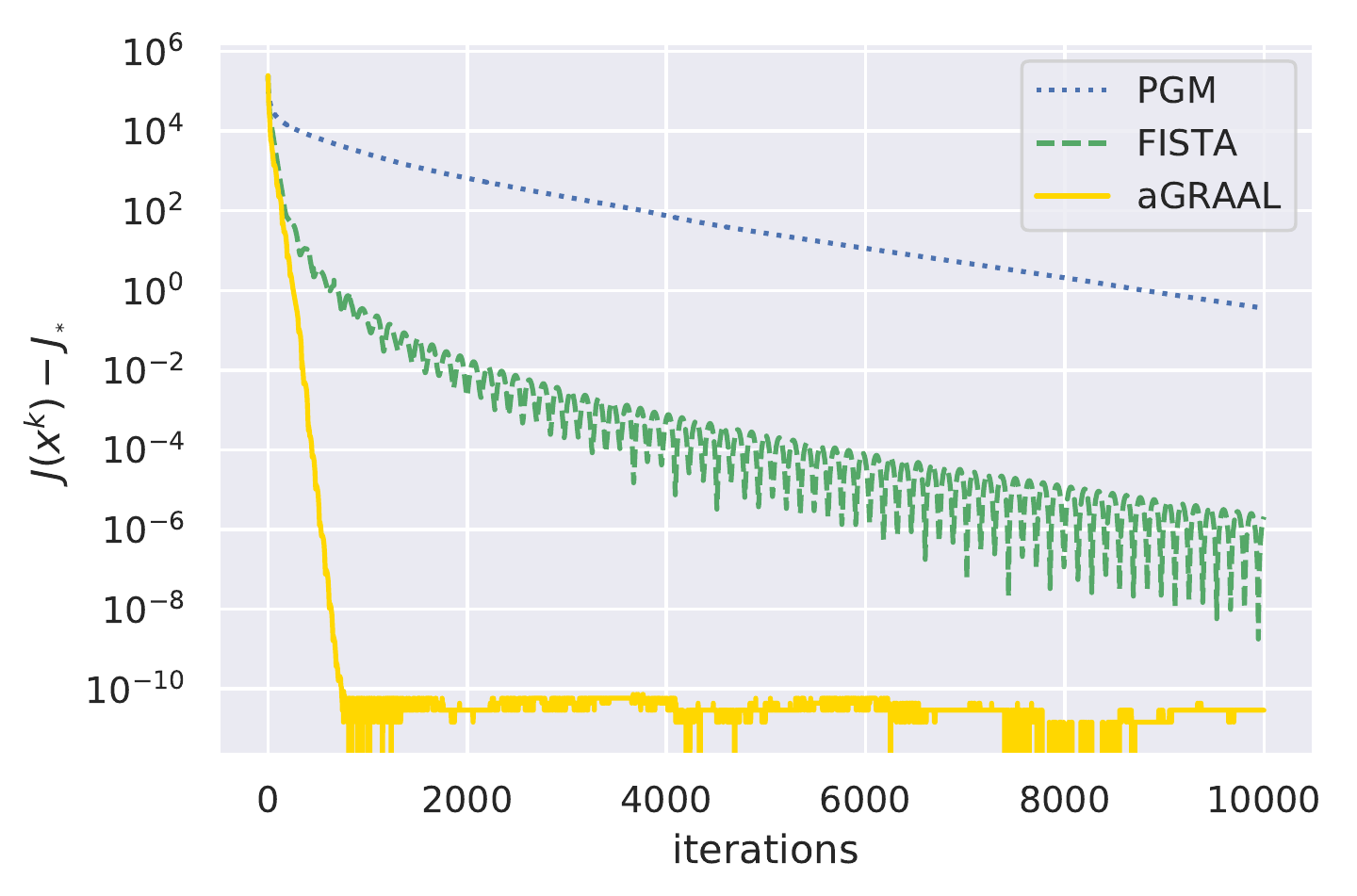}
    \end{subfigure}
\caption{ Results for problem \eqref{eq:slr}. Datasets \emph{real-sim}
(left) and \emph{kdda-2010} (right).}
\label{fig:regr}
\end{figure}

We have presented the results of only two test problems merely for
compactness, in fact similar results were observed for other datasets
that we tested: \emph{rcv1}, \emph{a9a}, \emph{ijcnn1},
\emph{covtype}. An explanation for such a good performance of aGRAAL
is of course that for this problem the global Lipschitz constant of
$\nabla f$ is too conservative. Notice also that our algorithm did not
take into account that in this case $F = \nabla f$ is a potential
operator. It would be interesting to see how we can enhance aGRAAL
with this information.

\subsection{Beyond monotonicity}
Finally, we illustrate numerically that aGRAAL can work even for
nonmonotone problems. 

\textbf{Nonmonotone equation.}  We would like to find a non-zero
solution of $F(z): = M(z)z = 0$, where
$M\colon \R^n\to \R^{n\times n}$ is a matrix-valued function. This is
of course a VI \eqref{vip} with $g\equiv 0$. We construct $M$ such
that $M(z)$ is a symmetric positive semidefinite matrix for any
$z$. Thus, $F$ obviously satisfies condition~\eqref{mvi}, since with
$\z = 0$ we have $\lr{M(z),z}\geq 0$ for all $z$. This also
indicates that $z = 0$ is a (trivial) solution of our problem.

For the experiment, we define $M$ as
\begin{equation}
    \label{eq:numF}
     M(z) := t_1t_1^T + t_2t_2^T, \quad \text{with}\quad t_1 = A\sin
    z, \quad t_2 = B \exp z,
\end{equation}
where $z\in \R^n$, $A,B \in \R^{n\times n}$ and the operations $\sin$ and $\exp$
should be understood to apply entrywise.  In fact, for our purposes
one term $t_1t_1^T$ in $F$ was sufficient, however we want to be
sure that a non-zero solution of our problem will not coincide with a
solution of a simple problem, like $Az = 0$.

For each $n\in \{100,500, 1000, 5000\}$ we generate $100$
random problems. We run aGRAAL for $10000$ iterations and stopped it
whenever the accuracy $\n{F(z^k)}\leq 10^{-6}$ reached.
Table~\ref{tab:bey} shows the success rate of solving this problem, i.e.,
we counted only those problems where $\n{z^k}$ was large enough to
make sure that this is not a trivial solution. We also report the
average number of iterations (among all successful instances) the
method needs to find a non-trivial solution. The entries of $A$, $B$
are drawn independently from the normal distribution
$\mathcal{N}(0,1)$. The starting point is always
$z^1 = (1,\dots, 1)$. It is clear that $F$ is not monotone; moreover,
the transcendental functions $\sin $, $\exp$ make this problem highly
nonlinear.

\begin{table}\centering
    \footnotesize
    \caption{Results of aGRAAL for nonmonotone problem
        \eqref{eq:numF}} \label{tab:bey}
\ra{1}
\begin{tabular}{@{}rrrrrr@{}}\toprule
  & &  $n=100$ &$n=500$ & $n=1000$ & $n=5000$  \\
  \midrule \\
  \addlinespace[-0.3cm]
  \texttt{iter}  && 526 & 614 & 667 & 1532\\
  \texttt{rate}  && 100 & 100 & 100 & 99 \\
  \bottomrule
\end{tabular}
\end{table}

\section{Conclusions and further directions}
\label{sec:concl}
We conclude our work by presenting some possible directions for future
research.

\textbf{Fixed point iteration.} It is interesting to represent scheme
\eqref{graal-1} as a fixed point iteration. To this end, let
\begin{equation}
    G =     \begin{bmatrix}
       \id & 0\\
       0 & \prox_{\la g}
    \end{bmatrix}
\qquad  R =     \begin{bmatrix}
        \frac 1 \p \id & \frac{\p-1}{\p} \id\\
        \frac 1 \p \id & \qquad \frac{\p-1}{\p} \id - \la F
    \end{bmatrix}.
\end{equation}
Then it is not difficult to see that we can rewrite \eqref{graal-1} as
\begin{equation}
    \binom{\z^k}{z^{k+1}} =
    (G \circ R)  \binom{\z^{k-1}}{z^k}.
\end{equation}
If we now set $\u^k = (\z^{k-1},z^{k})$, the above equation simplifies
to $ \u^{k+1} = GR\u^k$. However, so far it not clear how to derive
convergence of \eqref{graal-1} from the fixed point
perspective. Although, $G$ is a firmly-nonexpansive operator, $R$ is
definitely not; thus, it is difficult to say something meaningful
about $G\circ R$.

\textbf{Inertial extensions.} Starting from the paper~\cite{polyak},
it is observed that using some inertia for the optimization algorithm
often accelerates the latter. Later,
papers~\cite{attouch01,moudafi03,pock:inertial,boct2016inertial} extend this idea to a
more general case with monotone operators. In our case it will be in
particular interesting to do so, since our scheme
\begin{equation}
    z^{k+1} = \prox_{\la g}(\z^{k}-\la F(z^k))
\end{equation} 
uses $\z^k$ as a convex combination of all previous iterates
$z^1\dots, z^{k}$. This is completely opposite to the inertial
methods, where one uses $z^k + \a (z^k-z^{k-1})$ for some $\a > 0$.

\textbf{Bregman distance.} For many VI methods it is possible to
derive their analogues for the Bregman distances, as this is done for
the the extragradient method~\cite{korpel:76} by extensions
~\cite{nemirovski2004prox,nesterov2007dual}.  It is possible to do so
for GRAAL? This extension is not trivial since, for example, in
\eqref{eq:identity} we have used the identity \eqref{eq:id}, where the
linear structure was explicitly used.

\textbf{Stochastic settings.} For large-scale problems it is often the
case that even computing $F(z^k)$ becomes prohibitively expensive. For
this reason, the stochastic VI methods that compute $F(z^k)$
approximately can be advantageous over their deterministic
counterparts, as it was shown in
\cite{baes2013randomized,juditsky2011solving}.  It is interesting to
derive similar extensions for GRAAL. The same concerns the coordinate
extensions of GRAAL.

\textbf{Continuous dynamic.} Many discrete optimization/VI methods can
be studied from the continuous perspective which may shed light on the
discrete method. Originated from 60-s, this line of research was
popularized by many authors, see
\cite{antipin1994minimization,attouch1996dynamical,attouch2018fast,banert2015forward}
and references therein. This research brings new and often deeper
understanding of the respective iterative schemes, as, for instance,
the case with the Nesterov acceleration in
\cite{su2014differential,attouch2018fast}.
For aGRAAL it is a challenging task to derive a continuous scheme, since the
function $\la(t)$ cannot be defined in advance.

\small
\paragraph{Acknowledgements.}
Y.~Maltsky was supported by German Research Foundation grant
SFB755-A4.  The author would like to thank Anna--Lena Martins,
Panayotis Mertikopoulos, {Matt\-hew} Tam, associate editor and
anonymous referees for their useful comments that have significantly
improved the quality of the paper.

\bibliographystyle{acm}
\bibliography{biblio}
\end{document}